\documentclass[12pt]{article}
\usepackage{no-ejc}
\usepackage{mathtools}
\usepackage{ytableau}
\usepackage{kbordermatrix}

\usepackage{amsmath,amssymb}

\usepackage{graphicx}
\usepackage{tikz}
\usetikzlibrary{matrix}

\title{\bf Intersecting Families of Perfect Matchings}

\author{Nathan Lindzey\\
\small Department of Combinatorics and Optimization\\[-0.8ex]
\small University of Waterloo\\[-0.8ex] 
\small Waterloo, ON. Canada\\
\small\tt nlindzey@uwaterloo.ca
}

\dateline{}{}{}

\begin{document}
\maketitle

\begin{abstract}
A family of perfect matchings of $K_{2n}$ is \emph{t-intersecting} if any two members share $t$ or more edges.  
We prove for any $t \in \mathbb{N}$ that every $t$-intersecting family of perfect matchings has size no greater than $(2(n-t) - 1)!!$ for sufficiently large $n$, and that equality holds if and only if the family is composed of all perfect matchings that contain some fixed set of $t$ disjoint edges. This is an asymptotic version of a conjecture of Godsil and Meagher that may be viewed as the non-bipartite analogue of the \emph{Deza-Frankl conjecture} proven by Ellis, Friedgut, and Pilpel.
\end{abstract}

\section{Introduction}

\noindent Let $\mathcal{M}_{2n}$ be the collection of all perfect matchings of $K_{2n} = (V,E)$, the complete graph on an even number of vertices.  In this work, we investigate families of perfect matchings $\mathcal{F} \subseteq \mathcal{M}_{2n}$ that are \emph{t-intersecting}, that is, $|m \cap m'| \geq t$ for all $m,m' \in \mathcal{F}$.  In particular, we seek a characterization of the largest $t$-intersecting families of perfect matchings of $K_{2n}$.  The first candidates that come to mind are those families whose members all share a fixed set $T \subseteq E(K_{2n})$ of $t$ disjoint edges: 
\[ \mathcal{F}_T := \{ m \in \mathcal{M}_{2n} : T \subseteq m\}.\]
Such a family will be called \emph{canonically t-intersecting} for $t \geq 2$ and \emph{canonically intersecting} for $t=1$.  It is well-known that the canonically intersecting families are the largest intersecting families of perfect matchings of $K_{2n}$.
\begin{theorem}\label{thm:ekr}
\emph{\cite{GodsilM15, Lindzey17, MeagherM05}} If $\mathcal{F} \subseteq \mathcal{M}_{2n}$ is an intersecting family, then 
\[|\mathcal{F}|  \leq (2(n-1) - 1)!!.\] 
Moreover, equality holds if and only if $\mathcal{F}$ is a canonically intersecting family.
\end{theorem}
\noindent In their recent monograph on Erd\H{o}s-Ko-Rado combinatorics, Godsil and Meagher conjecture that a $t$-intersecting version of Theorem~\ref{thm:ekr} is true~\cite[Conjecture 16.14.2]{GodsilMeagher}. 
The main result of this work is that the following asymptotic version of their conjecture holds.  
\begin{theorem}[Main Result]\label{thm:tekr}
For any $t \in \mathbb{N}$, if $\mathcal{F} \subseteq \mathcal{M}_{2n}$ is a $t$-intersecting family, then 
\[|\mathcal{F}|  \leq (2(n-t) - 1)!!\] 
for sufficiently large $n$ depending on $t$. Moreover, equality holds if and only if $\mathcal{F}$ is a canonically $t$-intersecting family. 
\end{theorem} 
\noindent Our proof is similar in spirit to a few algebraic proofs of $t$-intersecting Erd\H{o}s-Ko-Rado results, including the somewhat recent proof of Deza and Frankl's conjecture on \emph{$t$-intersecting families of permutations}, equivalently, perfect matchings of the bipartite graph $K_{n,n}$.
\begin{theorem}\emph{\cite{EllisFP11,Ellis11}}\label{thm:perm}
For any $t \in \mathbb{N}$, if $\mathcal{F}$ is a $t$-intersecting family of perfect matchings of $K_{n,n}$, then
\[|\mathcal{F}|  \leq (n-t)!.\] 
for sufficiently large $n$ depending on $t$.
Moreover, equality holds if and only if $\mathcal{F}$ is a canonically t-intersecting family.
\end{theorem}
\noindent One may interpret our main result as the non-bipartite analogue of Theorem~\ref{thm:perm}. Viewing it as such, the most significant point of departure from the bipartite case is the absence of group structure on $\mathcal{M}_{2n}$, a sobering fact that will lead us through the theory of \emph{association schemes}, \emph{Gelfand pairs}, and \emph{symmetric functions}.

Before we begin this adventure, we first cover some preliminary material needed in order to map out the first part of our main result.

%

\section{Combinatorial Preliminaries and a Proof Sketch}

It is easy to see that perfect matchings are in one-to-one correspondence with partitions of $[2n] := \{1,2,\cdots,2n\}$ into parts of size two, so we may write any perfect matching $m \in \mathcal{M}_{2n}$ as a partition 
$$m = m_1~m_2|m_3~m_4|\cdots|m_{2n-1}~m_{2n} \text{ such that } m_i \in [2n].$$ 
Let $m^* := 1~2|3~4|\cdots |2n$-$1~2n$ be the \emph{identity perfect matching}. The \emph{symmetric group} $S_{2n}$ on $2n$ symbols acts transitively on $\mathcal{M}_{2n}$ under the following action:
\[ \sigma m = \sigma(m_1)~\sigma(m_2)~|~\sigma(m_3)~\sigma(m_4)~|~\cdots~|~\sigma(m_{2n-1})~\sigma(m_{2n}).\]
The \emph{hyperoctahedral group} $H_n := S_2 \wr S_n$ has order $(2n)!! := 2^nn!$ and is isomorphic to the stabilizer of $m^*$. Since perfect matchings are in one-to-one correspondence with cosets of the quotient $S_{2n}/H_n$, we see that 
$$|\mathcal{M}_{2n}| = (2n-1)!! := 1 \times 3 \times 5 \times \cdots \times (2n-3) \times (2n - 1).$$  
The following proposition can be shown using Stirling's formula.
\begin{proposition}\label{prop:sqrt}
\emph{\cite{Bauer07}} For all $n \in \mathbb{N}$, we have $(2n-1)!! < (2n)!!/\sqrt{\pi n}$. 
\end{proposition}
\noindent It will be convenient to let $(\!(2n-1)\!)_t := (2n-1) \times (2(n-1)-1) \times \cdots \times (2(n-t+1)-1)$ and $(\!(2n)\!)_t := 2n \times 2(n-1) \times \cdots \times 2(n-t+1)$ be the odd and even double factorial analogues of the \emph{falling factorial} $(n)_t := n!/t!$.

For any two perfect matchings $m,m' \in \mathcal{M}_{2n}$, let $\Gamma(m,m')$ be the multigraph on $[2n]$ whose edge multiset is the multiset union $m \cup m'$.  It is clear that $\Gamma(m,m') \cong \Gamma(m',m)$ is composed of disjoint cycles of even parity.  Let $k$ denote the number of disjoint cycles and let $2\lambda_i$ denote the length of an even cycle. If we order the cycles from longest to shortest and divide each of their lengths by two, we see that each graph corresponds to a (integer) \emph{partition} $\lambda = (\lambda_1, \lambda_2, \cdots , \lambda_k) \vdash n$. For any $\lambda \vdash n$, if there are $k$ parts that all have the same size $\lambda_i$ we use $\lambda_i^k$ to denote the multiplicity. Let $d(m,m'): \mathcal{M}_{2n} \times \mathcal{M}_{2n} \mapsto \lambda(n)$ denote this map where $\lambda(n)$ is the set of all partitions of $n$. We shall refer to $d(m,m')$ as the \emph{cycle type of $m'$ with respect to m} (or vice versa since $d(m,m') = d(m',m)$), and we say that $d(m^*,m)$ is \emph{the cycle type of m}.  Since $\Gamma(x,y) \cong \Gamma(x',y')$ if and only if $d(x,y) = d(x',y')$, let the graph $\Gamma_{\lambda}$ be a distinct representative from the isomorphism class $\lambda \vdash n$.  Illustrations of the graphs $\Gamma_{(n)}$ and $\Gamma_{(2,1^{n-2})}$ are provided in Figure~\ref{fig:graphs} for $n = 4$.

\begin{figure}
\centering
\begin{tikzpicture}[thick, scale=0.3]
  \foreach \x in {1,...,8}{
    \pgfmathparse{(\x)*(360 - 360/8)  + 7*360/16}
    \node[draw,circle,inner sep=0.1cm] (\x) at (\pgfmathresult:5.4cm) [ultra thick] {\textbf{\x}};
  }
     \draw[red,line width = 3] (1) -- (2);
     \draw[red,line width = 3] (3) -- (4);
     \draw[red,line width = 3] (5) -- (6);
     \draw[red,line width = 3] (7) -- (8);
     \draw[blue,dotted,line width = 3] (2) -- (3);
     \draw[blue,dotted,line width = 3](4) -- (5);
     \draw[blue,dotted,line width = 3] (6) -- (7);
     \draw[blue,dotted,line width = 3] (8) -- (1);
\end{tikzpicture}
\quad \quad \quad \quad \quad \quad 
\begin{tikzpicture}[thick, scale=0.3]
  \foreach \x in {1,...,8}{
    \pgfmathparse{(\x)*(360 - 360/8)  + 7*360/16}
    \node[draw,circle,inner sep=0.1cm] (\x) at (\pgfmathresult:5.4cm) [ultra thick] {\textbf{\x}};
  }
     \path[red,line width = 3] (1) edge[bend right] (2);
     \draw[red,line width = 3] (3) -- (4);
     \path[red,line width = 3] (5) edge[bend right] (6);
     \draw[red,line width = 3] (7) -- (8);
     \path[blue,dotted,line width = 3] (1) edge[bend left] (2);
     \draw[blue,dotted,line width = 3] (3) -- (8);
     \draw[blue,dotted,line width = 3] (4) -- (7);
     \path[blue,dotted,line width = 3] (5) edge[bend left] (6);
\end{tikzpicture}
\caption{The perfect matching $2~3|4~5|6~7|1~8$ on the left has cycle type $(n) \vdash n$ whereas the perfect matching  $1~2|3~8|4~7|5~6$ on the right has cycle type $(2,1^{n-2}) \vdash n$ for $n = 4$.}\label{fig:graphs}
\end{figure}
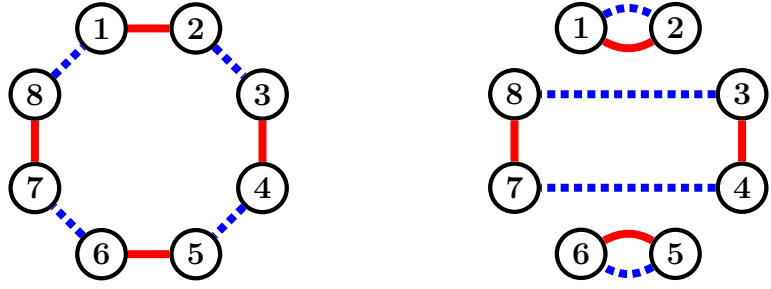

A \emph{$t$-derangement} of $\mathcal{M}_{2n}$ is a perfect matching $m \in \mathcal{M}_{2n}$ whose cycle type has fewer than $t$ parts of size 1. Their number, denoted as $D_2(n,t)$, can be determined via a recurrence quite similar to the classic one for permutation $t$-derangements:
\[D_2(0,1) = 1; D_2(1,1) = 0;\] 
\[D_2(n,1) = 2 (n - 1)(D_2(n - 1,1) + D_2(n - 2,1));\]
\[D_2(n,t) = \sum_{i=1}^t  \binom{2n}{2i}(2i-1)!! D_2(n-i,1).\]
Alternatively, we may compute $D_2(n,1)$ via the principle of inclusion-exclusion:
\[D_2(n,1) = \sum_{i=0}^n (-1)^i \binom{2n}{2i} (2(n-i)-1)!!.\]
Using this count, one can show that
\begin{align}\label{eq:derange}
D_2(n,1) = (2n-1)!!(1/\sqrt{e}+o(1)).
\end{align}

\noindent A \emph{pseudo-adjacency matrix} of a graph $\Gamma = (V,E)$ is a symmetric $|V| \times |V|$ matrix $\widetilde{A}(\Gamma)$ with constant row sum such that $\widetilde{A}(\Gamma)_{uv} \neq 0$ only if $uv \in E(\Gamma)$.
We let $\eta_i$ denote the eigenvalue associated to the $i$th eigenspace of a given pseudo-adjacency matrix, and we let $\eta_{\min} := \min_i \eta_i$ denote its least eigenvalue.  For any subgraph $\Gamma '$ of a graph $\Gamma = (V,E)$, let $V(\Gamma ') \subseteq V$ be the set of vertices of $\Gamma'$. An \emph{independent set} of $\Gamma$ is a subset $S \subseteq V(\Gamma)$ such that $uv \notin E(\Gamma)$ for all $u,v \in S$.  For any subset $X \subseteq V$, let $1_X$ denote its \emph{characteristic function}, that is, $1_X(v) = 1$ if $v \in X$; otherwise, $1_X(v) = 0$ for all $v \in V$.

The following result of Delsarte and Hoffman, which we shall call the \emph{ratio bound}, is the centerpiece of the first part of our main result.
\begin{theorem}[Ratio Bound]\label{thm:pseudoHoffman}
Let $\widetilde{A}(\Gamma)$ be a pseudo-adjacency matrix of $\Gamma = (V,E)$ with eigenvalues $\eta_1 \geq \eta_2 \geq \cdots \geq \eta_{\min}$ and corresponding orthonormal eigenvectors $v_1, v_2 \cdots , v_{\min}$. If $S\subseteq V$ is an independent set of $\Gamma$, then
\[ |S| \leq |V|\frac{-\eta_{\min}}{\eta_1 - \eta_{\min}}.\]
Moreover, if equality holds, then
\[ 1_S \in \emph{Span}\left( \{v_1\} \cup \{v_i : \eta_i = \eta_{\min} \} \right).\]
\end{theorem}
\noindent After writing $f := 1_S = \sum_{i=1}^{|V|}a_iv_i$ in the basis of orthonormal eigenvectors and setting $\alpha := |S|/|V|$, the ratio bound is easy to see once one observes
\begin{align}\label{eq:ratio}
0 = f^\top \widetilde{A}(\Gamma) f = \sum^{|V|}_{i=1} \eta_ia_i^2 &\geq \eta_1\alpha^2 + \eta_{\min}  \sum_{i = 2}^{|V|} a_i^2  = \eta_1\alpha^2 + 	\eta_{\min}(\alpha - \alpha^2),
\end{align}
where the last equality follows from Parseval's identity.  The ratio bound now follows from routine manipulation.

The graph that we will apply to the ratio bound is the \emph{perfect matching $t$-derangement graph} $\Gamma_t$ defined such that two perfect matchings $m,m' \in \mathcal{M}_{2n}$ are adjacent if they have less than $t$ edges in common, equivalently, $d(m,m')$ has fewer than $t$ parts of size 1. For $t=1$, we recover the \emph{perfect matching derangement graph}, which has received a fair amount attention in recent years~\cite{GodsilMeagher,GodsilM15,KuW17,Lindzey17}.  
It is easy to see that the independent sets of $\Gamma_t$ are in one-to-one correspondence with $t$-intersecting families of $\mathcal{M}_{2n}$; therefore, the ratio bound gives an upper bound on the size of a $t$-intersecting family.

Taking a quick glimpse ahead, in Section~\ref{sec:assoc} we show that the eigenspaces of $\Gamma_t$ and appropriately defined pseudo-adjacency matrices $\widetilde{A}(\Gamma_t)$ are parameterized by integer partitions $\lambda \vdash 2n$ that are \emph{even}, that is, 
$$\lambda = 2 \mu := (2\mu_1,2\mu_2, \cdots, 2\mu_k) \vdash 2n$$ 
for some $\mu \vdash n$.  Of these eigenspaces, the ones corresponding to \emph{fat even partitions}, $2\lambda \vdash 2n$ such that $2\lambda_1 \geq 2(n-t)$, will be of utmost importance.\\ 

\noindent We are now in a position to outline the proof of the upper bound of the main result.

\subsection*{Proof Sketch I}
Our goal is to show there exists a pseudo-adjacency matrix $\widetilde{A}(\Gamma_t)$ with eigenvalues $\eta_1 \geq \eta_2 \geq \cdots \geq \eta_{\min}$ satisfying
\[(2(n-t)-1)!! = (2n-1)!!\frac{-\eta_{\min}}{\eta_1 - \eta_{\min}}.\] 
By the ratio bound, such a matrix would imply that any canonically $t$-intersecting family $\mathcal{F}$ is a maximum independent set of $\Gamma_t$ and that any maximum independent set $S$ satisfies
\[ 1_S \in \text{Span}\left( \{v_1\} \cup \{v_i : \eta_i = \eta_{\min} \} \right)\]
where $v_i$ is an eigenvector corresponding to eigenvalue $\eta_i$. These two consequences imply
\[  \text{Span}\left\{ \{v_1\} \cup \{v_i : \eta_i = \eta_{\min} \} \right\} \geq \text{Span}\{ 1_{\mathcal{F}} : \mathcal{F} \text{ is canonically $t$-intersecting} \}.\]
We would like these two spaces to coincide, but the way the right-hand side is defined makes it particularly hard to determine which eigenspaces of our pseudo-adjacency matrix should correspond to the minimum eigenvalue. 
It turns out that for $t < n/2$, that the span of characteristic functions of canonically $t$-intersecting families is a subspace of the eigenspaces corresponding to fat even partitions, which we show in Section~\ref{sec:low}. 

In light of this, we define our pseudo-adjacency matrix $\widetilde{A}(\Gamma_t)$ so that $\eta_\lambda = \eta_{\min}$ for all fat even partitions $2\lambda$ except for $2\lambda = (2n)$.
In Section~\ref{sec:pseudo}, we show this is possible for sufficiently large $n$ by constructing a pseudo-adjacency matrix satisfying the foregoing such that $|\eta_\mu| = o(|\eta_{\lambda}|)$ for all even partitions $2\mu$ that are not fat, which will conclude the proof of the upper bound of our main result.  Our proof sketch of the characterization of the extremal families is deferred to Section~\ref{sec:sketch2}.

The strategy outlined above lies atop a somewhat baroque algebraic foundation, which we spend the next few sections developing.  

\section{Representation Theory Background}
Our work draws upon the fundamentals of the ordinary representation theory of finite groups and their Fourier analysis. Statisticians~\cite{CST,Diaconis88} have given a few treatments of group representation theory from a Fourier-analytical point of view, to which we refer the reader for more details. Throughout this section, let $H,K \leq G$ be subgroups of a finite group $G$, and $V$ be a finite dimensional vector space over $\mathbb{C}$.

\subsection{Finite Group Representation Theory}
For any set $X$, let $\mathbb{C}[X]$ denote the vector space of dimension $|X|$ of complex-valued functions over $X$, equipped with the inner product $$\langle f,g \rangle_X := \sum_{x \in X} f(x)\overline{g(x)}.$$  
A \emph{representation} $(\phi,V)$ of $G$ is a homomorphism $\phi : G \rightarrow GL(V)$ where $GL(V)$ is the \emph{general linear group}, that is, the group of $(\dim V) \times (\dim V)$ invertible matrices. It is customary to be less formal and denote the representation $(\phi,V)$ simply as $\phi$ when $V$ is understood, or as $V$ when $\phi$ is understood.  For any representation $\phi$, we define its \emph{dimension} to be $\dim \phi := \dim V$. When working concretely with a representation $\phi$, we abuse terminology and let $\phi(g)$ refer to a $(\dim \phi) \times (\dim \phi)$ matrix realization of $\phi$. Two representations $\rho,\phi$ are \emph{equivalent} if $\rho(g)$ and $\phi(g)$ are similar for all $g \in G$.

Let $(\phi,V)$ be a representation of $G$, and let $W \leq V$ be a \emph{G-invariant} subspace, that is, $\phi(g)w \in W$ for all $w \in W$ and for all $g \in G$.  We say that $(\phi|_W,W)$ is a \emph{subrepresentation} of $\phi$ where $\phi|_W$ is the restriction of $\phi$ to the subspace $W$. A representation $(\phi,V)$ is an \emph{irreducible representation} (or simply, an \emph{irreducible}) if it has no proper subrepresentations. The \emph{trivial representation} $(1, \mathbb{C})$ defined such that $1 : g \rightarrow 1$ for all $g \in G$ is clearly an irreducible of dimension one for any group $G$.

It is well-known that there is a bijection between the set of inequivalent irreducibles of $G$ and its conjugacy classes $\mathcal{C}$, and that any representation $V$ of $G$ decomposes uniquely as a direct sum of inequivalent irreducibles $V_i$ of $G$:
\[V \cong \bigoplus_{i=1}^{|\mathcal{C}|} m_iV_i\]
where $m_i$ is $V_i$'s \emph{multiplicity}, i.e., the number of occurrences of $V_i$ in the decomposition.

A natural way to find representations of groups is to let them act on sets.
In particular, for any group $G$ acting on a set $X$, let $(\phi,\mathbb{C}[X])$ be the \emph{permutation representation} of $G$ on $X$ defined such that
\[\phi(g)[f(x)] = f(g^{-1}x)\]
for all $g \in G$, $f \in \mathbb{C}[X]$, and $x \in X$.
\noindent If we let $G$ act on itself, then we obtain the \emph{regular representation}, which admits the following decomposition into irreducibles:
\[ \mathbb{C}[G] \cong \bigoplus_{i=1}^{|\mathcal{C}|}~(\dim V_i)~V_i\]
where $V_i$ is the $i$th irreducible of $G$.

Letting $e_ge_h = e_{gh}$ over the standard basis $\{e_g\}_{g \in G}$ of $\mathbb{C}[G]$, we see that $\mathbb{C}[G]$ is an algebra, the so-called \emph{group algebra} of $G$. Some of our results will involve a few calculations in the group algebra $\mathbb{R}[S_{2n}]$, whose decomposition into irreducibles admits an elegant combinatorial description overviewed in Section~\ref{sec:sym}.

For any $K \leq G$, there is a chain of subalgebras  $\mathbb{C}[K \backslash G /K] \leq \mathbb{C}[G/K] \leq \mathbb{C}[G]$ where
$$\mathbb{C}[G /K] = \{ f \in \mathbb{C}[G] : f(g) = f(gk)~\forall g \in G,~\forall k \in K \}$$ 
is the algebra of functions that are constant on the right cosets $G/K$, and 
$$\mathbb{C}[K \backslash G /K] = \{ f \in \mathbb{C}[G] : f(g) = f(kgk')~\forall g \in G,~\forall k,k' \in K \}$$  
is the algebra of functions that are constant on the double cosets $K \backslash G /K$.  The chain
$$\mathbb{R}[H_n \backslash S_{2n} /H_n] \leq \mathbb{R}[S_{2n}/H_n] \cong \mathbb{R}[\mathcal{M}_{2n}] \leq \mathbb{R}[S_{2n}]$$
will be of particular importance to us.  In Section~\ref{sec:assoc}, we show how $\mathbb{R}[\mathcal{M}_{2n}]$ decomposes into a direct sum of irreducibles of $S_{2n}$ by considering the permutation representation of $S_{2n}$ acting on $\mathcal{M}_{2n}$.

For any (irreducible) representation $\phi$ of $G$, the \emph{(irreducible) character} $\chi_\phi$ of $\phi$ is the map $\chi_\phi : G \rightarrow \mathbb{C}$ such that $\chi_\phi (g) := \text{Tr}(\phi(g))$. Similar matrices have the same trace, thus characters of representations are \emph{class functions}, that is, they are constant on conjugacy classes. Furthermore, the characters of the set of all irreducible representations of a group $G$ form an orthonormal basis for the space of all class functions of $\mathbb{C}[G]$.  

From these basic properties of characters, it is not hard to show the following.
\begin{lemma}\label{lem:mults}
\emph{\cite{Diaconis88}} Let $(\rho,V')$ be an irreducible of $G$ and let $(\phi,V)$ be a representation of $G$ such that
\[V \cong  \bigoplus_{i=1}^{|\mathcal{C}|} m_iV_i.\]
Then the number of irreducibles $V_i$ that are equivalent to $V'$ equals $\langle \chi_\phi,\chi_\rho \rangle = m_i$.
\end{lemma}

Lemma~\ref{lem:zerosum} is a generalization of the familiar fact that the sum of all of primitive $n$th roots of unity is zero. Its proof is essentially a corollary of Schur's lemma, and it will be helpful for simplifying sums of representations.
\begin{lemma}[Schur's Lemma]
Let $(\varphi,V)$ and $(\psi,W)$ be representations of $G$, and let $T : V \rightarrow W$ be a linear transformation. If $T\varphi(g) = \psi(g)T$ for all $g \in G$, then $T$ is either the zero map or an isomorphism.
In particular, if $(\phi,V)$ is an irreducible of $G$, then the only linear operators of $V$ that commute with $\phi$ are scalar multiples of the identity.
\end{lemma}

\begin{lemma}\label{lem:zerosum}
If $\rho$ is a non-trivial irreducible of $G$, then 
\[\sum_{g \in G} \rho(g) = 0.\] 
\end{lemma}
\begin{proof}
For any $k \in G$, we have
\begin{align*}
\rho(k) \big(\sum_{g \in G} \rho(g)\big) = \sum_{g \in G} \rho(k)\rho(g) &= \sum_{g' \in kG} \rho(g')\\
 &= \sum_{g' \in Gk} \rho(g')\\
 &=\big(\sum_{g \in G} \rho(g)\big)\rho(k),
\end{align*}
implying that $\sum_{g \in G} \rho(g) = cI$ for some constant $c$ by Schur's Lemma.  Since $\rho \neq 1$ is irreducible, it follows that $\langle \chi_\rho, \chi_{1} \rangle = 0$, thus $c=0$. We deduce that $\sum_{g \in G} \rho(g) = 0$.
\end{proof}
\noindent Lemma~\ref{lem:prod} shows that direct products of groups and their irreducibles behave as expected.  
\begin{lemma}\label{lem:prod}
\emph{\cite{Serre96}} Let $(\tau,V)$ and $(\rho,W)$ be irreducibles of finite groups $G$ and $G'$ respectively.  Then $V \otimes W$ is an irreducible of $G \times G'$, and any irreducible of $G \times G'$ is of this form.
\end{lemma}
It is well-known that any representation $\phi$ of $G$ is equivalent to a direct sum of irreducibles of $H \leq G$. This representation  $(\phi \downarrow^G_H ,V \downarrow^G_H)$ is called the \emph{restriction} of $\phi$ to $H$, and is obtained simply by restricting the domain of $\phi$ to $H$. Even if $\phi$ is an irreducible of $G$, the restricted representation is typically not an irreducible of $H$.  

Conversely, any representation of $H \leq G$ is equivalent to a direct sum of irreducibles of $G$, which we describe below.
Let $(\rho,V)$ be a representation of $H$ and $\hat{g}_1, \hat{g}_2, \cdots, \hat{g}_k$ be a system of distinct representatives (SDR) of $G/H$ where $k := [G:H]$. Define
\[  V \uparrow^G_H := \bigoplus_{i=1}^k \mathbf{\hat{g_i}}V\]
where $\mathbf{\hat{g}_i}V \cong V$ is a copy of $V$ associated to the coset $\hat{g_i}H$.  For any $g \in G$, there exists an $h_i \in H$ and $\hat{g}_{j(i)} \in G$ where $j(i) \in [k]$ such that $g^{-1}\hat{g_i} = \hat{g}_{j(i)}h_i$. Define $\rho \uparrow^G_H$ to act on  $V \uparrow^G_H$ as follows:
\[  \left( \rho \uparrow^G_H (g) \right) \sum_{i=1}^k \mathbf{\hat{g}_i}v_i = \sum_{i=1}^k \mathbf{\hat{g}_{j(i)}} \rho(h_i)v_i \]
such that $\mathbf{\hat{g}_i}v_i \in \mathbf{\hat{g_i}}V$ and $\mathbf{\hat{g}_{j(i)}} \rho(h_i)v_i \in \mathbf{\hat{g}_{j(i)}} V$ for all $g \in G$.  We say that $(\rho \uparrow_H^G , V \uparrow^G_H)$ is the \emph{induced representation} of $\rho$. It is easy to see that $\dim \left( V\uparrow^G_H \right) = k \cdot (\dim V)$, and we may compute the character of $\rho \uparrow^G_H$ as follows: 
\[ \chi_{\rho \uparrow^G_H }(g) = \sum_{x \in S} \chi_\rho(x^{-1}gx),\]
where $S$ is an SDR for $G/H$ and $\chi_\rho(x^{-1}gx)= 0$ if $x^{-1}gx \notin H$. Notice that $1 \uparrow^G_K$ is equivalent to the permutation representation of $G$ on $G/K$
\[  1 \uparrow^G_K (g) \sum_{i=1}^{k} e_{g_i} = \sum_{i=1}^k e_{g_{j(i)}},\]
and so it follows that $1 \uparrow^G_K$  and $\mathbb{C}[G/K]$ are isomorphic. 

It is clear that restriction is transitive, that is, $(\rho \downarrow^G_H) \downarrow^H_K \cong \rho \downarrow^G_K$ for any representation $\rho$ of $G$ such that $K \leq H \leq G$.  A less trivial fact is that the same is true of induction.
\begin{theorem}[Transitivity of Induction~\cite{Serre96}]
Let $\rho$ be a representation of $K$ such that $K \leq H \leq G$. Then 
\[ (\rho \uparrow_K^{H}) \uparrow_H^G \cong \rho \uparrow_K^G.\]
\end{theorem}
A discussion of group representations and their characters would not be complete without at least mentioning the following theorem of Frobenius and its corollary.
\begin{theorem}[Frobenius Reciprocity for Characters~\cite{Serre96}]
For any representation $\phi$ of $G$ and representation $\rho$ of $H$, we have
\[ \langle \chi_{\rho \uparrow^G_H }, \chi_\phi  \rangle = \langle \chi_{\phi \downarrow^G_H }, \chi_\rho \rangle.\]
\end{theorem}
\begin{corollary}
\emph{\cite{Serre96}} Let $\phi$ be an irreducible of $G$ and $\rho$ be an irreducible of $H$.  Then the the multiplicity of $\phi$ in $ \rho \uparrow^G_H $ equals the multiplicity of $\rho$ in $\phi \downarrow^G_H $.
\end{corollary}
\noindent Finally, the bare essentials of Fourier analysis will be needed for a few calculations. 

Let $f \in \mathbb{C}[G]$ and $\phi$ be an irreducible of $G$.  The \emph{Fourier transform} of $g$ is a matrix-valued function $\hat{f}$ on irreducible representations
\[  \hat{f}(\phi) = \frac{1}{|G|}\sum_{g \in G} f(g)\phi(g).\]
Letting $\hat{G}$ denote the complete set of irreducibles of $G$,  we may write any $f \in \mathbb{C}[G]$ as
\[  f(g) = \frac{1}{|G|}\sum_{\phi \in \hat{G}} (\dim \phi)~\text{Tr}[\phi(g^{-1})\hat{f}(\phi)].\]
Doing calculations in the Fourier basis for arbitrary non-Abelian groups is usually quite difficult, as it requires a very concrete understanding of the group's irreducibles, which the Reverend Alfred Young divined for the symmetric group at the turn of the last century. The next section provides no more than an illustrated glossary of his theory, and we refer the reader to~\cite{MacDonald95,StanleyV201} for first-rate introductions to the subject.

\subsection{Representation Theory of the Symmetric Group}\label{sec:sym}
For the remainder of this section, we let $\lambda$ be an integer partition of $n$.
We may visualize $\lambda$ as a \emph{Ferrers diagram}, a left-justified table of cells that contains $\lambda_i$ cells in the $i$th row. When referencing a Ferrers diagram, we alias $\lambda$ as the \emph{shape}.  For example, the Ferrers diagram below has shape $(5,3,2,1) \vdash 11$:
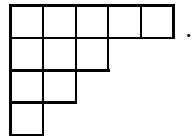
\begin{center}
\ytableausetup{mathmode,boxsize=1.0em}
\begin{ytableau}
\empty  & \empty &  \empty & \empty & \empty \\
\empty & \empty & \empty \\ 
 \empty & \empty  \\ 
 \empty \\  
\end{ytableau}~.
\end{center}

Let $\lambda' \vdash n$ denote the \emph{transpose} of $\lambda$, that is, the partition obtained by interchanging the columns and the rows of the corresponding Ferrers diagram of $\lambda$.  The transpose of $(5,3,2,1) \vdash 11$ is $(4,3,2,1,1) \vdash 11$, as illustrated below:
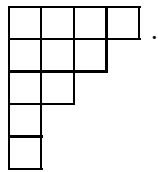
\begin{center}
\ytableausetup{mathmode,boxsize=1.0em}
\begin{ytableau}
\empty  & \empty &  \empty & \empty \\
\empty & \empty & \empty \\ 
 \empty & \empty  \\ 
 \empty \\  
  \empty \\ 
\end{ytableau}~.
\end{center}
Let $\ell(\lambda)$ denote the \emph{length} of $\lambda$, that is, the number of rows in its Ferrers diagram. 
We say that $\lambda$ \emph{covers} $\mu$ if $\mu_i \leq \lambda_i$ for each $i \in [\ell(\mu)]$.  If $\lambda$ and $\mu$ are two partitions such that $\lambda$ covers $\mu$, then we obtain the \emph{skew shape} \emph{$\lambda/\mu$} by removing the cells corresponding to $\mu$ from $\lambda$.  For instance, the shape $(5,3,2,1)$ covers $(2,2,1)$, so we may consider the skew shape $(5,3,2,1)/(2,2,1)$:
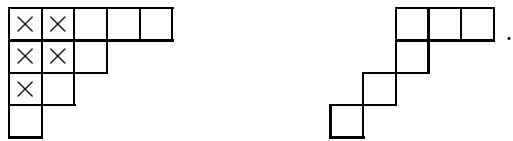
\begin{center}
\quad\quad\quad\quad\quad\quad\quad\quad\quad\quad\quad\quad
\ytableausetup{mathmode,boxsize=1.0em}
\begin{ytableau}
\times  & \times &  \empty & \empty & \empty \\
\times & \times & \empty\\ 
 \times & \empty  \\ 
 \empty \\  
\end{ytableau}
\hfill
\ytableausetup{mathmode,boxsize=1.0em}
\begin{ytableau}
\none & \none &  \empty & \empty & \empty \\
\none & \none & \empty\\ 
 \none & \empty  \\ 
 \empty \\  
\end{ytableau}~.
\quad\quad\quad\quad\quad\quad\quad\quad\quad\quad\quad\quad
\end{center}
A skew shape is a \emph{horizontal strip} if each column has no more than one cell.  For example, the skew shape $(5,3,2,1)/(3,2,1)$ is a horizontal strip.

The only order on the set of partitions $\lambda(n)$ considered in this work is the \emph{reverse-lexicographical order} $(\lambda(n),\leq)$, defined such that $\mu \leq \lambda$ if and only if $\mu_j < \lambda_j$ where $j$ is the first index where $\mu$ and $\lambda$ differ, or $\mu = \lambda$.  For $n = 4$, we have
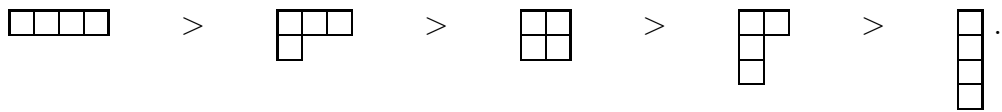
\[
\ytableausetup{mathmode,boxsize=0.75em}
\begin{ytableau}
\empty & \empty & \empty & \empty\\
\end{ytableau}
\quad\quad
 >
\quad\quad
\begin{ytableau}
\empty & \empty & \empty\\
\empty \\
\end{ytableau}
\quad\quad
>
\quad\quad
\begin{ytableau}
\empty & \empty \\
\empty & \empty \\
\end{ytableau}
\quad\quad
>
\quad\quad
\begin{ytableau}
\empty & \empty \\
\empty \\
\empty\\
\end{ytableau}
\quad\quad
>
\quad\quad
\begin{ytableau}
\empty \\
\empty \\
\empty\\
\empty\\
\end{ytableau}~.
\]
The following is an interlacing sequence of definitions and examples concerning \emph{Young tableaux}, which we simply refer to as \emph{tableaux} (not to be irreverent to the Reverend). 

A \emph{generalized $\lambda$-tableau} $T$ is a set of $n$ cells arranged in $k$ left-justified ordered rows such that ordered row $i$ has $\lambda_i$ cells and each cell is labeled with a number of $[n]$. 

\begin{center}
\ytableausetup{mathmode,boxsize=1.2em}
\begin{ytableau}
1  & 9 & 7 & 3 & 1 \\
8 & 6 & 6\\ 
5 & 2  \\ 
2 \\  
\end{ytableau}
\end{center}

A \emph{$\lambda$-tableau} $T$ is a set of $n$ cells arranged in $k$ left-justified ordered rows such that ordered row $i$ has $\lambda_i$ cells and each cell is assigned a unique number of $[n]$. 

\begin{center}
\ytableausetup{mathmode,boxsize=1.2em}
\begin{ytableau}
1  & 2 & 7 & 8 & 9 \\
4 & 10 & 5\\ 
3 & 6  \\ 
11 \\  
\end{ytableau}
\end{center}

A \emph{standard $\lambda$-tableau} $T_\lambda$ is a $\lambda$-tableau with entries strictly increasing along  rows and strictly increasing along columns.

\begin{center}
\ytableausetup{mathmode,boxsize=1.2em}
\begin{ytableau}
1 & 2 & 5 & 8 & 9 \\
3 & 6 & 7\\ 
4 & 10  \\ 
11 \\  
\end{ytableau}
\end{center}

 A \emph{semistandard $\lambda$-tableau} $T_\lambda$ is a generalized $\lambda$-tableau where the numbers are weakly increasing along the rows and strictly increasing along the columns. The multiplicity of each number is its \emph{weight}, and the weights of $1,2,\cdots,n$ are recorded as a $n$-tuple $\mu = (\mu_1, \mu_2, \cdots , \mu_k)$ where $\mu_i$ is the number of times that the number $i$ occurs.  Below is a semistandard $(5,3,2,1)$-tableau with weight $(1,2,4,3,1)$.
\[ 
\ytableausetup{mathmode,boxsize=1.2em}
\begin{ytableau}
1 & 2 & 3 & 3 & 3 \\
2 & 3 & 4\\ 
4 & 4 \\ 
5 \\  
\end{ytableau}
\]
A \emph{semistandard $\lambda/\mu$-tableau} is a generalized tableau of skew shape $\lambda/\mu$ where the numbers are weakly increasing along rows and strictly increasing along columns.

Let $f^\lambda$ denote the number of standard $\lambda$-tableaux, and let $K_{\lambda,\mu}$ denote the number of semistandard tableaux of shape $\lambda$ and weight $\mu$.  The $K_{\lambda,\mu}$'s are called the \emph{Kostka numbers} and we will meet a generalization of them in Section~\ref{sec:char}.
We may count the number of standard Young tableau of shape $\lambda$ elegantly as follows.

Let $s \in \lambda$ be a cell of a Ferrers diagram of shape $\lambda$. Define $h_\lambda(s) := a_{\lambda}(s) + l_{\lambda}(s) + 1$ where $l_{\lambda}(s)$ is the number of cells on the same column below $s$ and $a_{\lambda}(s)$ is the number of cells on the same row to the right of $s$. Denote the \emph{hook product} as $h(\lambda) := \prod_{s \in T}  h_{\lambda}(s)$.

\begin{theorem}[The Hook Rule \cite{MacDonald95}]$f^\lambda = n!/h(\lambda)$.
\end{theorem}

It turns out that the \emph{tabloids} are actually a reliable source of information for the representation theory of symmetric group. A \emph{$\lambda$-tabloid} $\{T\}$ is a collection of $n$ cells, arranged in $k$ left-justified unordered rows, such that unordered row $i$ has $\lambda_i$ cells and each of the $n$ cells is labeled by a unique number of $[n]$.  To emphasize the lack of order along the rows, we draw them like so
\begin{center}
\ytableausetup{mathmode,boxsize=1.2em,tabloids = true}
\begin{ytableau}
1 & 2 & 7 & 8 & 9 \\
4 & 10 & 5\\ 
3 & 6  \\ 
11 \\  
\end{ytableau}~.
\end{center}

\noindent The notation $\{T\}$ suggests that each $\lambda$-tabloid is an equivalence class of $\lambda$-tableaux, which is indeed the case. For any $\lambda$-tableau $T$, let $\{T\}$ be the tabloid that $T$ belongs to. Let $\mathbb{T}_\lambda$ denote the set of all $\lambda$-tabloids. Note that $S_n$ acts on the entries of the cells of $\lambda$-tableaux and $\lambda$-tabloids in the obvious way. 

For any $\lambda$-tableau $T$, let $C_T \leq S_n$ be the \emph{column-stabilizer of T}, i.e., the permutation group that fixes the columns of $T$ setwise. In particular, if $T$ has shape $\lambda$, then we may write the column-stabilizer as $C_T \cong S_{(\lambda')_1} \times S_{(\lambda')_2} \times \cdots \times S_{(\lambda')_{\lambda_1}}$ where $S_{(\lambda')_j}$ acts on the row indices of the $j$th column of $T$, that is,
\[  \sigma t_{i,j} = t_{\sigma(i),j}   \]
for all $\sigma \in S_{(\lambda')_j}$ and cells $t_{i,j} \in T$ in the $j$th column of $T$. Having $C_T$ act on the indices of cells rather than their entries will be useful for a few calculations in that arise in the proof of a key lemma.

Let $(\phi, \mathbb{R}[\mathbb{T}_\lambda])$ be the permutation representation of $S_n$ acting on $\mathbb{R}[\mathbb{T}_\lambda]$ with the standard basis $\{e_{\{T\}} : \{T\} \in \mathbb{T}_\lambda \}$ in the natural way. We briefly discuss how this representation decomposes into irreducibles, as several objects therein will resurface in Section~\ref{sec:combchar}.

For each $\lambda$-tableau $T$, define the \emph{T-polytabloid} to be the following ($\pm 1$)-linear combination of $\lambda$-tabloids 
$$e_T :=  \sum_{\pi \in C_T} \text{sign}(\pi) e_{\{\pi T\}}.$$  
To give an example, if we let
\[
T = 
\ytableausetup{mathmode,boxsize=1.2em,tabloids=false}
\begin{ytableau}
1 & 2 & 3 & 6 & 7 \\
4 & 5\\ 
\end{ytableau}~,\]
then its polytabloid is just
\[
e_T = 
\ytableausetup{mathmode,boxsize=1.2em,tabloids=true}
\begin{ytableau}
1 & 2 & 3 & 6 & 7 \\
4 & 5\\ 
\end{ytableau}
-
\begin{ytableau}
4 & 2 & 3 & 6 & 7 \\
1 & 5\\ 
\end{ytableau}
-
\begin{ytableau}
1 & 5 & 3 & 6 & 7 \\
4 & 2\\ 
\end{ytableau}
+
\begin{ytableau}
4 & 5 & 3 & 6 & 7 \\
1 & 2\\ 
\end{ytableau}~.
\]
It is clear that for each $\lambda \vdash n$ the permutation representation of $S_{n}$ acting on the set 
$$\{ e_T : T \text{ is a }\lambda\text{-tableau}\}$$ 
is a subrepresentation of $(\phi, \mathbb{R}[\mathbb{T}_\lambda])$. Specht showed this is in fact an \emph{irreducible} of $S_n$ and that $\{ e_T : T \text{ is a standard }\lambda\text{-tableau}\}$ forms a basis for 
$$S^\lambda := \text{Span}\{ e_T : T \text{ is a }\lambda\text{-tableau}\}.$$ 
It follows that $S^\lambda$ has dimension $f^\lambda$, which we state as a theorem. 
\begin{theorem}[Dimension Formula~\cite{MacDonald95}]
$\dim S^\lambda = f^\lambda$.
\end{theorem}

Since two elements are conjugate in $S_n$ if and only if they have the same cycle type, each irreducible identifies with an integer partition of $n$.  Throughout this work, we frequently abuse notation by letting $\lambda$ refer to $S^\lambda$, which should not result in any confusion.

We end this section with some non-standard combinatorial terminology for tableaux and a few bounds on the dimensions of irreducibles of the symmetric group. Recall that 
$$2\lambda := (2\lambda_1, 2\lambda_2, \cdots ,2\lambda_k) \vdash 2n$$
denotes an even partition of $2n$, and that $\lambda(n)$ denotes the set of all partitions of $n$. Let $2\lambda(n)$ be the set of all even partitions of $\lambda(2n)$.

Following Ellis, Friedgut, and Pilpel~\cite{EllisFP11}, for any $t < n/2$, we say that $\lambda \vdash n$ is \emph{fat} if $\lambda \geq (n-t,1^t)$, \emph{tall} if $\lambda' \geq (n-1,1^t)$, or \emph{medium} otherwise. For any $t < n/2$, we say that an even partition $\lambda \vdash 2n$ is \emph{fat} if $\lambda \geq 2(n-t,1^t)$, and \emph{non-fat} otherwise. It is easy to see that there are no tall even partitions for $t < n/2$.
\begin{proposition}
For any $t < n/2$, there is no $\mu \in 2\lambda(n)$ such that $\mu'$ is fat.
\end{proposition}
Another simple fact is that for any $t < n/2$, the number of fat partitions, which we denote as $F_t$, does not depend on $n$.
\begin{proposition}
For any $t < n/2$, the number of fat partitions of $\lambda(n)$ is a constant $F_t$ depending only on $t$, and $F_t$ is also equal to the number of fat partitions of $2\lambda(n)$.
\end{proposition}
Lemma~\ref{lem:medDims} is a lower bound on the dimension of non-fat irreducibles that follows immediately from the proof of~\cite[Lemma 2]{EllisFP11}.\footnote{To avoid confusion, we note a typo in the statement of \cite[Lemma 2]{EllisFP11}: ``... greater than $n-k$ ..." should be ``... less than $n-k$ ...".}
\begin{lemma}\label{lem:medDims}
\emph{\cite{EllisFP11}}
For any $t \in \mathbb{N}$, there exists a constant $C_t > 0$ depending only on $t$ such that $\dim 2\lambda \geq C_t (2n)^{2(t+1)}$ for any non-fat even partition $2\lambda \vdash 2n$.
\end{lemma}
We conclude this section with a classical upper bound on the dimension of an irreducible of the symmetric group.
\begin{theorem}\label{thm:dimbound}
\emph{\cite{Diaconis88}} For any irreducible $\lambda$ of $S_n$, we have $\dim \lambda \leq \binom{n}{\lambda_1} \sqrt{(n-\lambda_1)!}$. 
\end{theorem}

\section{Association Schemes and Gelfand Pairs}\label{sec:assoc}

The theory of association schemes will be a convenient suitcase for packing the algebraic and combinatorial components of our work. We refer the reader to Bannai and Ito's reference~\cite{BannaiI84} or Godsil and Meagher's manuscript~\cite{GodsilMeagher} for a more thorough treatment.

A \emph{symmetric association scheme} is a collection of $m+1$ binary $|X| \times |X|$ matrices $A_0,A_1,\cdots,A_m$ over a set $X$ that satisfy the following axioms:
\begin{enumerate}
\item $A_i$ is symmetric for all $0 \leq i \leq m$,
\item $A_0 = I$ where $I$ is the identity matrix,
\item $\sum_{i=0}^m A_i = J$ where $J$ is the all-ones matrix, and 
\item $A_iA_j  = A_jA_i  \in \text{Span}\{A_0, A_1,\cdots, A_m\}$ for all $0 \leq i,j \leq m$.
\end{enumerate}
\noindent From a combinatorial perspective, the \emph{associates} $A_1, \cdots , A_m$ are in fact adjacency matrices of regular spanning subgraphs of $K_{|X|}$ that partition $E(K_{|X|})$ subject to other regularity conditions that need not concern us. 

The matrix algebra generated by the identity matrix and its associates is the association scheme's \emph{Bose-Mesner algebra}, and these matrices form a basis for this algebra. What gives symmetric association schemes their charm is that their Bose-Mesner algebras afford a canonical dual basis of \emph{primitive idempotents}, positive semi-definite matrices $E_0,E_1,\cdots,E_m$ that are pairwise-orthogonal and satisfy $\sum_{i=0}^m E_i = I$.  In particular, the $i$th primitive idempotent is the projector corresponding to the $i$th eigenspace of any matrix in the Bose-Mesner algebra.  

Any zero-diagonal matrix that lies in the Bose-Mesner algebra of a symmetric association scheme is a pseudo-adjacency matrix of a regular graph that satisfies the hypotheses of the ratio bound. However, for this bound to be useful, one must have some understanding of the eigenvalues of the associates that compose it.  Such information is recorded in the association scheme's  \emph{character table}, a $(m+1) \times (m+1)$ matrix $P$ whose rows are indexed by eigenspaces, columns indexed by associates, and defined such that $P_{ij}$ is the eigenvalue of the $i$th eigenspace of $A_j$.  

For symmetric association schemes that arise from groups, the entries of $P$ can be described in terms of the spherical functions of a finite symmetric \emph{Gelfand pair}~\cite{BannaiI84}.  
\begin{theorem}\label{thm:gelfandPair}
\emph{\cite{MacDonald95}} Let $K \leq G$ be a group.  Then the following are equivalent.
\begin{enumerate}
\item $(G,K)$ is a Gelfand Pair;
\item The induced representation $1 \uparrow_K^G  \cong \bigoplus_{i=1}^m V_i$ (equivalently, the permutation representation of $G$ acting on $G/K$) is multiplicity-free;
\item The double-coset algebra $\mathbb{C}[K \backslash G /K]$ is commutative.
\end{enumerate}
Moreover, a Gelfand pair is symmetric if $KgK = Kg^{-1}K$ for all $g \in G$.
\end{theorem}
\noindent Let $(G,K)$ be a Gelfand pair, $X := G/K$, and define $\chi_i$ to be the character of $V_i$ as in the second statement of Theorem~\ref{thm:gelfandPair}, with dimension $d_i := \chi_i(1)$. The functions $\omega^1, \omega^2, \cdots,\omega^m \in \mathbb{C}[X]$ defined such that  
\[  \omega^i(g) = \frac{1}{|K|} \sum_{k \in K} \chi_i (g^{-1}k) \quad \forall g \in G\]
are called the \emph{spherical functions} and form an orthogonal basis for $\mathbb{C}[K \backslash G /K]$, equivalently, the algebra of (left) $K$-invariant functions of $\mathbb{C}[X]$. For any two indices $i,j$, define $\delta_{i,j} = 1$ if $i = j$; otherwise, $\delta_{i,j} = 0$.
\begin{proposition}\label{prop:orthogonal}\emph{\cite{CST}}
The spherical functions form a basis for the space of (left) $K$-invariant functions in $\mathbb{C}[X]$ and satisfy the orthogonality relations
\[\langle \omega^i,\omega^j\rangle_X =  \sum_{x \in X} \omega^i(x)\overline{\omega^j(x)} = \delta_{i,j}\frac{|X|}{d_i}.\]
\end{proposition}

The (left) $K$-orbits of $X$ partition the cosets into $(K\backslash G /K)$-double cosets, or so-called \emph{spheres} $\Omega_0, \Omega_1, \cdots, \Omega_m$. It is helpful to think of spheres and spherical functions as the spherical analogues of conjugacy classes and irreducible characters respectively.  Indeed, the spherical functions are constant on spheres, and it can be shown that the number of distinct spherical functions equals the number of distinct irreducibles of $\mathbb{C}[X]$, equivalently, the number of spheres of $X$~\cite{CST}.  We write $\omega^i_j$ for the value of the spherical function $\omega^i$ corresponding to the $i$th irreducible on the double coset corresponding to $\Omega_j$.

For any choice of $K \leq G$, a general procedure is given in~\cite{BannaiI84} for constructing a (not necessarily commutative) association scheme whose Hecke algebra is isomorphic to $\mathbb{C}[K \backslash G /K]$.  An association scheme $\mathcal{A}$ that arises from this construction will be called a \emph{$(K \backslash G / K)$-association scheme}.  In such a scheme, there is a natural bijection between the associates of $\mathcal{A}$ and the double cosets, and if $\mathcal{A}$ is a $(K \backslash G / K)$-association scheme, then $\mathcal{A}$ is symmetric if and only if $(G,K)$ is a symmetric Gelfand pair~\cite{BannaiI84}.  

The following theorem is a representation-theoretic characterization of the eigenvalues of any graph that arises from a commutative $(K \backslash G / K)$-association scheme.
\begin{theorem}\label{thm:eigs}\emph{\cite{Lindzey17,GodsilMeagher}}
Let $\Gamma = \sum_{j \in \Lambda} A_{j}$ be a sum of associates in a $(K \backslash G / K)$-association scheme such that $(G,K)$ is a Gelfand pair and $\Lambda$ is the index set of some subset of the associates. The eigenvalue $\eta_i$ of $\Gamma$ corresponding to the $i$th irreducible representation of $1\uparrow_K^G $ has multiplicity $d_i$ and can be written as
\[ \eta_i = \sum_{j \in \Lambda} | \Omega_j | \omega^i_{j}. \]
\end{theorem}
\begin{proposition}\label{prop:lessThanOne}
Let $\lambda,\mu \vdash n$. Then $| \omega^{\lambda}_\mu | \leq 1.$
\end{proposition}
\begin{proof}
Suppose there exists a $\lambda,\mu$ such that $|\omega^\lambda_\mu| > 1$.  Then by Theorem~\ref{thm:eigs}, the $\mu$-associate of $\mathcal{A}$ has an eigenvalue with magnitude is greater than its row sum $|\Omega_\mu|$, contradicting the Perron-Frobenius Theorem.
\end{proof}
\noindent The following lemma is a crude but useful upper bound on the eigenvalues of such graphs.
\begin{lemma}\label{lem:eigsbound}
Let $\Gamma = \sum_{j \in \Lambda} A_{j}$ be a sum of associates in a $(K \backslash G / K)$-association scheme such that $(G,K)$ is a Gelfand pair, $X = G/K$, and $\Lambda$ is the index set of some subset of the associates.  Then for $i$th irreducible representation of $1\uparrow_K^G $ we have
\[ |\eta_i| \leq \sqrt{ |X||\Omega_{\Lambda}|/d_i},\]
where $\Omega_\Lambda = \cup_{j \in \Lambda} \Omega_j$ is a disjoint union of spheres indexed by $\Lambda$.
\end{lemma}
\begin{proof}
Let $f \in \mathbb{C}[X]$ be the characteristic function of $\Omega_\Lambda$. Theorem~\ref{thm:eigs} implies that
$$\eta_i = \langle \omega^i , f \rangle_X,$$
By Proposition~\ref{prop:orthogonal}, we have $\langle \omega^i , \omega^i \rangle_X  = |X|/d_i$ for any spherical function $\omega^i$, so by the Cauchy-Schwarz inequality we see that
\[ |\eta_i| = |\langle \omega^i ,  f\rangle_X | \leq \sqrt{\langle \omega^i , \omega^i \rangle_{X} \langle f, f \rangle_X} = \sqrt{|X||\Omega_{\Lambda}|/d_i},\]
as desired.
\end{proof}

It is well-known that $(S_{2n}, H_n)$ is a symmetric Gelfand pair (see~\cite[Chap. VII]{MacDonald95}), and so the permutation representation of $S_{2n}$ acting on $\mathcal{M}_{2n}$, equivalently $1 \uparrow^{S_{2n}}_{H_n}$, admits a multiplicity-free decomposition into irreducibles as follows.
\begin{theorem}
\emph{\cite{Thrall42}}\label{thm:decomp} Let $\lambda = (\lambda_1, \lambda_2,\cdots,\lambda_k) \vdash n$ and let $2\lambda$ denote the irreducible of $S_{2n}$ corresponding to the partition $2\lambda = (2\lambda_1, 2\lambda_2,\cdots,2\lambda_k) \vdash 2n$. Then
\[ 1 \uparrow^{S_{2n}}_{H_n} \cong \bigoplus_{\lambda \vdash n}~2\lambda.\]
\end{theorem}

\noindent The corresponding symmetric $(H_n \backslash S_{2n} / H_n)$-association scheme is defined as follows. For each $\lambda \vdash n$, define the \emph{$\lambda$-associate} as the following $|\mathcal{M}_{2n}| \times |\mathcal{M}_{2n}|$ matrix,
\[
    (A_\lambda)_{i,j} = 
\begin{cases}
    1,& \text{if } d(i,j) = \lambda\\
    0,              & \text{otherwise}
\end{cases}
\]
where $i,j \in \mathcal{M}_{2n}$.  Let $\mathcal{A}$ be the set of all $\lambda$-associates, which we call \emph{the perfect matching association scheme}. For example, when $n = 4$ we have $\mathcal{A} = \{A_{(4)} , A_{(3,1)} , A_{(2,2)} , A_{(2,1^2)} , A_{(1^4)} \}$, and its character table is
\[
  P = \kbordermatrix{
    & (4) & (3,1) & (2,2) & (2,1^2) & (1^4) \\
    (4) & 48 & 32 & 12 & 12 & 1 \\
    (3,1) & -8 & 4 & -2 & 5 & 1 \\
    (2,2) & -2 & -8 & 7 & 2 & 1 \\
    (2,1^2) & 4 & -2 & -2 & -1 & 1 \\
    (1^4) & -6 & 8 & 3 & -6 & 1
  }.
\]
For a more detailed discussion of the perfect matching association scheme, see~\cite{GodsilMeagher}. 

The perfect matching $t$-derangement graph $\Gamma_t$ is the sum of $\lambda$-associates of $\mathcal{A}$ that have less than $t$ parts of size one.  Its eigenspaces are precisely the irreducibles of $\mathbb{R}[\mathcal{M}_{2n}]$ given in Theorem~\ref{thm:decomp}, and its eigenvalues are completely determined by Theorem~\ref{thm:eigs}.

For each $\lambda \vdash n$, the \emph{$\lambda$-sphere} is the following set:
\[ \Omega_\lambda := \{m \in \mathcal{M}_{2n} : d(m^*,m) = \lambda\} \]
where $\lambda \vdash n$.  The $\lambda$-spheres partition $\mathcal{M}_{2n}$, and as mentioned before, play the role of conjugacy classes in our spherical setting. Indeed, the following proposition is reminiscent of the elementary formula for the number of permutations in a given conjugacy class. 

For any $\lambda \vdash n$, let $m_i$ denote the number of parts of $\lambda$ that equal $i$, and define $$z_\lambda := \prod_{i \geq 1} i^{m_i} m_i!.$$
\begin{proposition}\label{prop:sphereSize} 
\emph{\cite{MacDonald95}}
For any $\lambda \vdash n$, we have $|\Omega_\lambda| = (2n)!!/(2^{\ell(\lambda)}z_\lambda)$.
\end{proposition}

\noindent Lemma~\ref{lem:spheres} follows from~\cite[{Lemma 5}]{EllisFP11}, but we include a proof for sake of completeness. 
\begin{lemma}\label{lem:spheres}
\emph{\cite{EllisFP11}} Let $k < n/2$. If $\Omega_\lambda$ is a sphere such that $\lambda$ has a part of size $n-k$, then
\[\frac{2^nn!}{2(n-k)(2k)^k} \leq |\Omega_\lambda | \leq 2^{n+1}(n-1)!\]
\end{lemma}
\begin{proof}
Let $l := \ell(\lambda)$. Note that $\lambda_2 + \lambda_3 + \cdots + \lambda_{l} = k$, thus $l-1 \leq k$.
By Proposition~\ref{prop:sphereSize}, we have 
\[ |\Omega_\lambda| = \frac{2^nn!}{2(n-k)2^{l-1}z_{\lambda \setminus \lambda_1}}.\]
Using the arithmetic mean-geometric mean inequality (AM/GM), we have
\begin{align*}
n/2 < n-k &< 2(n-k)2^{l-1}z_{\lambda \setminus \lambda_1}\\ 
&\leq 2(n-k)2^{l-1}(l-1)!\prod^l_{i=2}\lambda_i\\
&\leq 2(n-k)2^{l-1}(l-1)!\left( \frac{1}{(l-1)} \right)^{l-1} \quad \text{(AM/GM)}\\
&\leq 2(n-k)(2k)^{l-1}\\
&\leq 2(n-k)(2k)^k, 
\end{align*}
and so we arrive at
\[ \frac{2^nn!}{2(n-k)(2k)^k} \leq |\Omega_\lambda| \leq 2^{n+1}(n-1)!,\]
which completes the proof.
\end{proof}

\section{The Fourier Support of Canonically $t$-Intersecting Families}\label{sec:low}

The main result of this section is that the characteristic functions of canonically $t$-intersecting families of $\mathcal{M}_{2n}$ are supported on the ``even low frequencies" of the Fourier spectrum of $S_{2n}$.  More precisely, for any $t < n/2$, let $U_t$ be the space of functions of $\mathbb{R}[\mathcal{M}_{2n}]$ supported on the fat even partitions, that is,
\[ U_t = \{ f \in \mathbb{R}[\mathcal{M}_{2n}] : \widehat{f}(\rho) = 0 \text{ for all }\rho < 2(n-t,1^t)  \}.\]

\begin{theorem}\label{thm:supp}
For any $t < n/2$, we have 
$$\emph{Span}\{ 1_{\mathcal{F}} \in \mathbb{R}[\mathcal{M}_{2n}] : \mathcal{F} \subseteq \mathcal{M}_{2n} \emph{ is canonically $t$-intersecting} \} \leq U_t.$$
\end{theorem}
Before we prove this theorem, some preliminaries are in order. For any set $\mathcal{F} \subseteq \mathcal{M}_{2n}$, recall its characteristic function $1_{\mathcal{F}} \in \mathbb{R}[\mathcal{M}_{2n}]$ is defined such that $1_{\mathcal{F}}(m) = 1$ if $m \in \mathcal{F}$; otherwise, $1_{\mathcal{F}}(m) = 0$. Since every $m \in \mathcal{M}_{2n}$ corresponds to a right coset $\sigma H_n$ for some $\sigma \in S_{2n}$, there is a natural isomorphism between $\mathbb{R}[\mathcal{M}_{2n}]$ and the algebra of right $H_n$-invariant functions of $\mathbb{R}[S_{2n}]$, that is 
$$\mathbb{R}[\mathcal{M}_{2n}] \cong \{ f \in \mathbb{R}[S_{2n}] : f(\sigma) = f(\sigma h)~\forall \sigma \in S_{2n},~\forall h \in H_n \} \leq \mathbb{R}[S_{2n}].$$
Let $\widetilde{f} \in \mathbb{R}[S_{2n}]$ denote the $H_n$-invariant function of the group algebra of $S_{2n}$ corresponding to $f \in \mathbb{R}[\mathcal{M}_{2n}]$ under this isomorphism, and for any set $\mathcal{F} \subseteq \mathcal{M}_{2n}$, let $\widetilde{\mathcal{F}} \subseteq S_{2n}$ denote the corresponding set of permutations of size $|\mathcal{F}|(2n)!!$ that is a union of right cosets.  Under this isomorphism, it also follows that
\[ \langle \widetilde{f},\widetilde{f'}\rangle_{S_{2n}} = (2n)!!\langle f,f' \rangle_{\mathcal{M}_{2n}}\]
for any $f,f' \in \mathbb{R}[\mathcal{M}_{2n}]$. For any canonically $t$-intersecting family $\mathcal{F}_T$ we have
\[ |\widetilde{\mathcal{F}_T}| = (2(n-t)-1)!! (2n)!! = (2(n-t))! 2^t (n)_t.\]
Let $K := \text{Stab}_{S_{2n}}(T) \cong S_{2(n-t)} \times H_t$ be the stabilizer of $T$ with respect to $S_{2n}$. Recall that $T \subseteq m$ for each $m \in \mathcal{F}_T$, thus  $K \subseteq \widetilde{\mathcal{F}_T}$; however, these are not the only permutations that keep the edges of $T$ together. This can be seen by observing
\[ \frac{|\widetilde{\mathcal{F}_T}|}{|K|} = \frac{(2(n-t))! 2^t (n)_t}{(2(n-t))!2^tt!} = \binom{n}{t},\]
which suggests the following proposition that is not hard to see.
\begin{proposition}\label{prop:charcanon}
If $\mathcal{F}_T \subseteq \mathcal{M}_{2n}$ be canonically $t$-intersecting family, then its corresponding characteristic $H_n$-invariant function of $\mathbb{R}[S_{2n}]$ can be written as 
$$1_{\widetilde{\mathcal{F}_T}} = \sum_{s \in S} 1_{sK}$$
where $S$ is a set of $\binom{n}{t}$ representatives of distinct cosets of $S_{2n}/K$ and $1_{sK} \in \mathbb{R}[S_{2n}]$ is the characteristic function of the corresponding coset.
\end{proposition}

\begin{lemma}\label{lem:triv}
Let $\nu \vdash 2n$ be an irreducible of $S_{2n}$ and let $$K :=(S_{2(n-t)} \times H_t ) \leq (S_{2(n-t)} \times S_{2t}) =: H \leq S_{2n}.$$ Then the multiplicity of the trivial representation $1_K$ in $\nu \downarrow^{S_{2n}}_K $ equals the number of partitions $\mu \vdash t$ such that the Young diagram of $\nu \vdash 2n$ can be obtained from the Young diagram of $2\mu$ by adding $2(n - t)$ cells, no two in the same column.
\end{lemma}
\begin{proof} For any group $H$, let $1_H$ denote the trivial representation of $H$. By Lemma~\ref{lem:mults}, the multiplicity of $1_K$ in $\nu \downarrow^{S_{2n}}_K $ is $\langle \chi_{1_K,} \chi_{ \nu \downarrow^{S_{2n}}_{K}  }\rangle$.  We have
\begin{align*}
\langle \chi_{1_K,} \chi_{ \nu \downarrow^{S_{2n}}_{K}  }\rangle  &= \langle \chi_{1_K \uparrow^{S_{2n}}_K}, \chi_\nu \rangle &\text{(Frobenius Reciprocity)}\\
&= \langle \chi_{1_{S_{2(n-t)}} \otimes 1_{H_t} \uparrow^{H}_K  \uparrow^{S_{2n}}_{H} }, \chi_\nu \rangle &\text{(Lemma~\ref{lem:prod} \& Transitivity of $\uparrow$)}\\
 &= \langle \chi_{(2(n-t)) \otimes 1_{H_t} \uparrow^{S_{2t}}_{H_t}  \uparrow^{S_{2n}}_{H} }, \chi_\nu \rangle \\
 &= \langle \chi_{ \bigoplus_{\mu \vdash t} ((2(n-t)) \otimes (2\mu)) \uparrow^{S_{2n}}_{H} }, \chi_\nu \rangle &\text{(Theorem~\ref{thm:decomp})}\\
  &= \langle \chi_{ \bigoplus_{\mu \vdash t} (((2(n-t)) \otimes (2\mu))\uparrow^{S_{2n}}_{H})}, \chi_\nu \rangle&\text{(Linearity of $\uparrow$)}\\
    &= \sum_{\mu \vdash t} \langle \sum_{\lambda} \chi_\lambda, \chi_\nu \rangle &\text{(Pieri's Rule)}
\end{align*}
where $\sum_\lambda$ ranges over partitions $\lambda \vdash 2n$ obtainable from $2\mu$ by adding $2(n-t)$ cells, no two in the same column.  The result now follows from Lemma~\ref{lem:mults}.
\end{proof}
\begin{corollary}\label{cor:trivMult}
Let $\nu \vdash 2n$ be a non-fat even partition.  Then the trivial representation of $K$ does not occur in $\nu \downarrow^{S_{2n}}_K$.
\end{corollary}
\begin{proof}
Since $\nu$ is non-fat, its first row has less than $2(n - t)$ cells. But then $\nu$ cannot be obtained by adding $2(n - t)$ cells, no two in the same column, to any $\mu \vdash t$.
\end{proof}
\begin{proof}[Proof of Theorem~\ref{thm:supp}]
Let $f := 1_{\widetilde{\mathcal{F}_T}} \in \mathbb{R}[S_{2n}]$ be the characteristic function of a canonically $t$-intersecting family $\mathcal{F}_T \subseteq \mathcal{M}_{2n}$. Proposition~\ref{prop:charcanon} implies that $f = \sum_{s \in S} 1_{sK}$ where $K$ is the stabilizer of $T$ in $S_{2n}$ and $S$ is a set of $\binom{n}{t}$ coset representatives.  

Let $\rho$ be an even non-fat irreducible.  Applying the Fourier transform gives us
\begin{align*}
\widehat{f}(\rho) &= \frac{1}{(2n)!} \sum_{\sigma \in S_{2n}} f(\sigma)\rho(\sigma)\\
&= \frac{1}{(2n)!} \sum_{\sigma \in \widetilde{\mathcal{F}_T}} \rho(\sigma)\\
&= \frac{1}{(2n)!}\sum_{s \in S}\rho(s) \left( \sum_{k \in K} \rho(k) \right)   \\
&= \frac{1}{(2n)!}\sum_{s \in S} \rho(s) \left( \sum_{k \in K} \rho \downarrow^{S_{2n}}_{K}(k) \right) .\\
\intertext{By Corollary~\ref{cor:trivMult}, the trivial representation does not appear in $\rho \downarrow^{S_{2n}}_{K}$; therefore, writing $\rho \downarrow_K^{S_{2n}}$ as a direct sum of irreducibles and applying Lemma~\ref{lem:zerosum} gives us}
&= \frac{1}{(2n)!} \sum_{s \in S}  \rho(s) \left(  \sum_{k \in K} \rho \downarrow^{S_{2n}}_{K}(k) \right) = 0,
\end{align*}
which completes the proof.
\end{proof}

The following lemma shows that there is a canonically $t$-intersecting family whose characteristic function has nonzero Fourier weight on the irreducible $(2(n-t,1^t))$.
\begin{lemma}\label{lem:fatnonzero} If $S = \{\{3,4\},\{5,6\},\cdots,\{2t+1,2t+2\}\}$, then $\widehat{1_{\mathcal{F}_{S} }}(2(n-t,1^t)) \neq 0$.
\end{lemma}
\begin{proof}
Let $T$ be the unique standard Young tableau of shape $2(n-t,1^t)$ such that the second row of $T$ is $\{3,4\}$, the third row of $T$ is $\{5,6\}$, and so on. Define $1_{\{ T\}} \in \mathbb{R}[\mathcal{M}_{2n}]$ such that $1_{\{  T\}}(m) = 1$ if the endpoints of each edge of $m$ both exist in the same row of $\{ T\}$; otherwise, $1_{\{  T\}}(m) = 0$. 
Let 
$$f_{T} = \sum_{\sigma \in C_T} \text{sign}(\sigma) 1_{\{ \sigma T\}},$$ 
which lives in the $2(n-t,1^t)$ irreducible subspace of $\mathbb{R}[\mathcal{M}_{2n}]$ (see~\cite[Ch.11]{CST} or Section~\ref{sec:combchar} for a proof). For each $m \in \mathcal{F}_{S}$, we have 
$1_{\mathcal{F}_S}(m) \cdot \text{sign}(\sigma)1_{\{\sigma T\}}(m) \neq 0$ if and only if $\sigma = \sigma_1 \sigma_2$ where $\sigma_1$ and $\sigma_2$ are disjoint permutations that act on the cells of first and second columns respectively of $T$ in the same way. Any such $\sigma$ is even, so we have $\langle 1_{\mathcal{F}_{S}}, f_T \rangle_{\mathcal{M}_{2n}} > 0$.  Since the projection of $1_{\mathcal{F}_{S}}$ onto the irreducible $2(n-t,1^t)$ is nonzero, we have that $\widehat{1_{\mathcal{F}_{S} }}(2(n-t,1^t)) \neq 0$, as desired.
\end{proof}
%

\section{Character Theory and Symmetric Functions}\label{sec:char}

We now describe the character table of the perfect matching association scheme in terms of symmetric functions and their transition matrices. This point of view will give compact proofs of a few results needed to show that a pseudo-adjacency matrix of the perfect matching $t$-derangement graph with the correct eigenvalues exists.  We make no attempt to overview the theory of symmetric functions, but the majority of the material in this section can be found in MacDonald's text~\cite{MacDonald95}.

The characters of the symmetric group arise naturally in the study of \emph{Schur symmetric functions} $\{s_\mu\}_{\mu \vdash n}$, a well-known basis of the space of symmetric functions in $n$ variables.
In particular, when the \emph{power-sum symmetric functions} $\{p_\lambda\}_{\lambda \vdash n}$ are expressed in terms of the Schur functions, we obtain the irreducible characters of $S_n$:
\[ p_\lambda = \sum_{\mu \vdash n} \chi_\lambda^\mu s_\mu.\]
In other words, $M(p,s)$ is the character table of $S_n$ where $M(x,y)$ denotes the reverse-lexicographically ordered transition matrix from basis $x$ to basis $y$.  Similarly, when the power-sum symmetric functions are expressed in the \emph{monomial symmetric function} basis, we get the \emph{permutation characters}:  
\[ p_\lambda = \sum_{\mu \vdash n} D_{\lambda,\mu} m_\mu,\]
where $D_{\lambda, \mu}$ is equal to the number of ordered partitions $\pi = (B_1,\cdots, B_{\ell(\mu)})$ of the set $\{1,2,\cdots,\ell(\lambda)\}$ such that $\mu_j = \sum_{i \in B_j} \lambda_i$ for all $1 \leq j \leq \ell(\mu)$ (see~\cite{StanleyV201} for a proof).

It will be instructive to first give a short proof of~\cite[Theorem 20]{EllisFP11} from this viewpoint.  We begin by recalling a few well-known results.
\begin{theorem}
\emph{\cite{MacDonald95}} The transition matrix $M(p,m) = (D_{\lambda, \mu})$ is lower-triangular
\end{theorem}
\begin{theorem}\emph{\cite{MacDonald95}} The transition matrix $M(m,s) = (K_{\lambda, \mu})^{-1}$ is upper-unitriangular.
\end{theorem}
\begin{theorem}
An invertible matrix admits an $LU$-decomposition if and only if all its leading principal minors are nonsingular.
\end{theorem}
\noindent These theorems provide an easy proof of the following.
\begin{theorem}
Any leading principal minor of the character table of the symmetric group is invertible.
\end{theorem}
\begin{proof}
The character table of $S_n$ is a transition matrix, thus it is invertible. Its $LU$-decomposition is $L = (D_{\lambda , \mu})$ and $U = (K_{\lambda, \mu})^{-1}$.
\end{proof}
\noindent The leading principal minors relevant to us are the ones induced by all the fat partitions except for the skinniest fat partition $(n-t,1^t)$.  Define $F := F_t-1$ where $t < n/2$. MacDonald observed that such minors exhibit a ``stability" property (not to be confused with the notion of stability in extremal combinatorics).  Let $D(n)$ and $K(n)$ be the transition matrices $M(p,m)$ and $M(s,m)$ indexed by partitions $\lambda \vdash n$ in reverse-lexicographic order.
\begin{lemma}\emph{\cite{MacDonald95}}\label{lem:kostkaStable}
The $F \times F$ leading principal minor of $K(n)$ (resp. $K(n)^{-1}$) equals the $F \times F$ leading principal minor of $K(n')$ (resp. $K(n')^{-1}$) for all $n' \geq n$.
\end{lemma}
\noindent Essentially the same combinatorial argument as described in~\cite[pg. 105]{MacDonald95} gives rise to an analogous result for $D(n)$, surely known to MacDonald, but also proven in~\cite{EllisFP11}.
\begin{lemma}\label{lem:permStable}
The $F \times F$ leading principal minor of $D(n)$ (resp. $D(n)^{-1}$) equals the $F \times F$ leading principal minor of $D(n')$ (resp. $D(n')^{-1}$) for all $n' \geq n$.
\end{lemma}
\begin{corollary}
The $F \times F$ leading principal minor of the character table of $S_n$ equals the $F \times F$ leading principal minor of the character table of $S_{n'}$ for all $n' \geq n$.
\end{corollary}
\noindent The following result is now immediate.
\begin{theorem}\emph{\cite[Theorem 20]{EllisFP11}}
The $F \times F$ leading principal minor of the character table of $S_n$ is invertible with entries uniformly bounded by a function of $t$.
\end{theorem}
We seek a similar theorem for the zonal spherical analogue of characters, the so-called \emph{zonal characters}. Define the \emph{zonal character table} to be the $\lambda(n) \times \lambda(n)$ matrix $(\omega^\lambda_\rho)$ such that the $(\lambda,\rho)$-entry is given by $\omega^\lambda_\rho$. Such characters arise naturally as coefficients of the \emph{normalized zonal polynomials} $Z_\lambda'$, but first we should first introduce the unnormalized \emph{zonal polynomials} $Z_\lambda$, yet another basis for the space of symmetric functions that enjoys many of the same properties as the Schur symmetric functions.  

Zelevinsky credits Stanley with the observation that the coefficients of $Z_\lambda$ expressed in the power-sum basis are the $\lambda$-eigenvalues of the associates $A_\rho$ of the perfect matching association scheme $\mathcal{A}$~\cite[pg. 413]{MacDonald95}, which one may compare to Theorem~\ref{thm:eigs}: 
\begin{align*}
 Z_\lambda = |H_n| \sum_{\rho \vdash n} z_{2\rho}^{-1}\omega_\rho^\lambda p_\rho  = \sum_{\rho \vdash n} |\Omega_\rho| \omega_\rho^\lambda p_\rho.
 \end{align*} 
In other words, the coefficients of the zonal polynomials are precisely the entries of the character table of $\mathcal{A}$, which is in turn a column-normalization of $(\omega^\lambda_\rho)$. 
In particular, by inverting, normalizing, and setting $Z_\lambda' :=  \frac{|H_n|}{h(2\lambda)}Z_\lambda$, we obtain
 \begin{align*}
 p_\rho =  \sum_{\lambda \vdash n} \frac{|H_n|}{h(2\lambda)} \omega_\rho^\lambda Z_\lambda = \sum_{\lambda \vdash n} \omega_\rho^\lambda Z_\lambda'.
 \end{align*}
 Since column and row normalization does not affect invertibility, we can easily deduce the following two results.
\begin{theorem}\label{thm:minors}
Every leading principal minor of the zonal character table $(\omega_\rho^\lambda)$ is invertible with each entry no greater than one.
\end{theorem}
\begin{proof}
The transition matrix $M(m,Z)$ is upper unitriangular~\cite[pg. 408]{MacDonald95}. The $LU$-decomposition of $M(p,Z)$ is $L=(D_{\lambda, \mu})$ and $U = M(m,Z)$.  Since $(\omega_\rho^\lambda)$ is equivalent to $M(p,Z)$ up to normalization, it follows that $(\omega_\rho^\lambda)$ is an invertible matrix that admits an $LU$-decomposition.  We deduce that the leading principal minors of $(\omega_\rho^\lambda)$ are invertible, and Proposition~\ref{prop:lessThanOne} shows the magnitudes of its entries are no greater than one.
\end{proof}
\begin{corollary}
Each leading principal minor of $\mathcal{A}$'s character table is invertible.
\end{corollary}

We conclude this section by showing that the transition matrix $M(Z',m)$ enjoys a ``stability" property akin to Lemma~\ref{lem:kostkaStable}, which requires a brief digression into the theory of \emph{Jack symmetric functions}~\cite[pg. 376]{MacDonald95}.

It turns out that the Schur symmetric functions form the unique basis of the space of symmetric functions such that the transition matrix $M(s,m)$ is upper-unitriangular 
$$s_\lambda := m_\lambda + \sum_{\mu < \lambda}K_{\lambda \mu}m_\mu,$$ 
and the members of $\{s_\lambda\}$ are pairwise orthogonal with respect to the inner product 
$$\langle p_\lambda , p_\mu \rangle := 1^{l(\lambda)} z_\lambda \delta_{\lambda \mu}$$ 
where $\delta_{\lambda \mu} = 1$ if $\lambda = \mu$; otherwise, $\delta_{\lambda \mu} = 0$.

Similarly, the normalized zonal polynomials $\{Z_\lambda'\}$ form the unique basis of the space of symmetric functions such that the transition matrix $M(Z',m)$ is upper-unitriangular $Z_\lambda' := m_\lambda + \sum_{\mu < \lambda}K^{(2)}_{\lambda \mu}m_\mu$ and the members of $\{Z_\lambda'\}$ are pairwise orthogonal with respect to the inner product $\langle p_\lambda , p_\mu \rangle := 2^{l(\lambda)} z_\lambda \delta_{\lambda \mu}$.

Continuing in this manner, Jack showed that for any $\alpha \in \mathbb{R}$, there is a unique basis $\{P_\lambda\}$ for the space of symmetric functions that satisfies the following:
\begin{enumerate}
\item The transition matrix $M(P,m)$ is upper-unitriangular: $$P_\lambda := m_\lambda + \sum_{\mu < \lambda}K^{(\alpha)}_{\lambda, \mu}m_\mu, \text{ and }$$
\item the members of $\{P_\lambda\}$ are pairwise orthogonal with respect to the inner product $$\langle p_\lambda , p_\mu \rangle := \alpha^{l(\lambda)} z_\lambda \delta_{\lambda \mu}.$$
\end{enumerate}

We are interested in bounding entries of the leading $F \times F$ principal minor of the matrix $(K^{(\alpha)}_{\lambda,\mu})$ for $\alpha = 2$. MacDonald gives a combinatorial rule for computing these entries~\cite{MacDonald95}, which we describe below.  First, for any cell $s \in \lambda$, we define
\[ b_{\lambda}^{(\alpha)}(s) := \frac{\alpha a_\lambda(s) + l_\lambda(s) + 1}{\alpha a_\lambda(s) + l_\lambda(s) + \alpha},\]
which is less than 1 for all $\alpha > 1$.  Let
\[ \Psi_{\lambda / \mu}^{(\alpha )} := \prod_{s \in R_{\lambda / \mu} - C_{\lambda / \mu}} \frac{b_\mu^{(\alpha)}(s)}{b_\lambda^{(\alpha)}(s)}\]
where $C_{\lambda / \mu}$ (resp. $R_{\lambda / \mu}$) is the union of the columns (resp. rows) that intersect $\lambda - \mu$ (here, columns and tableaux are treated as sets whose elements are cells).
Then let
\[ \Psi_T^{(\alpha)} :=  \prod^r_{i=1} \Psi_{\lambda^{(i)} / \lambda^{(i-1)}}^{(\alpha)}, \]
where $0 = \lambda^{(0)} \subset \lambda^{(1)} \subset \cdots \subset \lambda^{(r)}$ is the sequence of partitions determined by the tableau $T$ so that each skew diagram $\lambda^{(i)} - \lambda^{(i-1)}$ is a horizontal strip (contains at most one cell in any column).  Finally, we have
\[ K^{(\alpha)}_{\lambda \mu} = \sum_T \Psi_T^{(\alpha)} \]
where $T$ ranges over semistandard tableaux of shape $\lambda$ and weight $\mu$.  Observe that when $\alpha = 1$, these are simply the Kostka numbers.  It is now simple matter to deduce that the Kostka numbers upperbound the $\alpha$-Kostka numbers.
\begin{lemma}\label{lem:kostkaStable2}
If $\alpha \geq 1$, then $K_{\lambda,\mu} \geq K^{(\alpha)}_{\lambda,\mu}$ for all $\lambda,\mu \vdash n$.
\end{lemma}  
\begin{proof}
Since $K_{\lambda,\mu} = K^{(\alpha)}_{\lambda,\mu}$ for $\alpha = 1$, we may assume that $\alpha > 1$.  It suffices to show that $\Psi_{\lambda / \mu}^{(\alpha )} < \Psi_{\lambda / \mu}^{(1)} = 1$.  Since $\lambda$ covers $\mu$, we have $b_{\lambda}^{(\alpha)}(s) > b_{\mu}^{(\alpha)}(s)$, implying that 
\[ \Psi_{\lambda / \mu}^{(\alpha )} = \prod_{s \in R_{\lambda / \mu} - C_{\lambda / \mu}} \frac{b_\mu^{(\alpha)}(s)}{b_\lambda^{(\alpha)}(s)} < 1,\]
as desired.
\end{proof}
\noindent Lemmas~\ref{lem:kostkaStable2} and~\ref{lem:kostkaStable} now imply the following.
\begin{corollary}\label{cor:bound}
The $(i,j)$-entry of the $F \times F$ leading principal minor of $K^{(2)}(n)$ is less than or equal to the $(i,j)$-entry of the $F \times F$ leading principal minor of $K(n')$ for all $n' \geq n$ and $1 \leq i,j \leq F$.  Moreover, the magnitudes of the entries of the $F \times F$ leading principal minor of $K^{(2)}(n)$ are bounded above by a function of $t$ for all $n$.
\end{corollary}

\section{Constructing the Pseudo-Adjacency Matrix}\label{sec:pseudo}

Having completed the legwork, we now show that for sufficiently large $n$, there exists a pseudo-adjacency matrix of $\Gamma_t$ that meets the ratio bound with equality, thereby proving the first part of the main result. Define
$$\zeta := -\frac{1}{(\!(2n-1)\!)_t-1}.$$
It suffices to show there exists a pseudo-adjacency matrix $\widetilde{A}(\Gamma_t)$ whose eigenvalues satisfy
\[\eta_{\min}/\eta_1 = \zeta,\]
as this would imply via the ratio bound that $|S| \leq (2(n-t)-1)!!$ for any independent set $S$ of $\Gamma_t$. 
To this end, we first show that for any $t \in \mathbb{N}$ there exists a pseudomatrix $\widetilde{A}(\Gamma_t)$ that has the desired eigenvalues on the fat even partitions. 
We then show that the magnitudes of the eigenvalues of $\widetilde{A}(\Gamma_t)$ corresponding to non-fat even partitions are smaller than the eigenvalues of the fat even partitions for sufficiently large $n$. 
\begin{theorem}\label{thm:main}
Let $A_{1}, A_{2}, \cdots ,A_{{F}}$ be the first $F = F_t -1$ associates of $\mathcal{A}$ in reverse-lexicographic order. Then there exists a constant $B_t > 0$ depending only on $t$ and a linear combination of fat associates
\[\widetilde{A}(\Gamma_t) := \sum_{j=1}^{F} x_j A_j\] 
such that $\max_j |x_j| \leq B_t/(2n-2)!!$, with eigenvalues $\eta_1,\eta_2,\cdots,\eta_F$ satisfying
\[\eta_i = \sum_{j=1}^{F} x_jP_{i,j} = \begin{cases} 
      1 & \emph{if } i=1 \\
      \zeta & \emph{if } 1 < i \leq F. 
   \end{cases}\]
\end{theorem}
\begin{proof}
For any matrix $A$ indexed by integer partitions of $n$, let $A_F$ denote its leading principal minor induced by all the fat partitions except for $(n-t,1^t)$ The entries of $M := P_{F}$ correspond to eigenvalues of the fat eigenspaces excluding $(n-t,1^t)$ of the first $F$ associates. Let $x = (x_1,\cdots,x_F)$ and $b := (1,\zeta,\cdots,\zeta)$.  
By Theorem~\ref{thm:minors}, we have that $M$ is invertible, thus $x$ is the unique solution to $Mx=b$.  
To bound $x$, observe that
\begin{align*}
x = M^{-1}b &= (M(p,Z')_F (\Omega_\lambda)_F)^{-1}b \\
&= (\Omega_\lambda)_F^{-1} M(p,Z')_F^{-1}b \\
&= (\Omega_\lambda)_F^{-1} M(Z',m)_F (D(n))_F^{-1}b.
\end{align*}
By Lemma~\ref{lem:permStable}, the magnitudes of the entries of $(D(n))_F^{-1}$ are upperbounded by some function of $t$.  By Corollary~\ref{cor:bound}, the magnitudes of the entries of $M(Z',m)_F$ are upperbounded by some function of $t$. By Lemma~\ref{lem:spheres}, we deduce that
\[|\Omega_\lambda |^{-1} \leq \frac{2(n-t)(2t)^t}{2^nn!} \leq \frac{B'_t}{(2n-2)!!},\]
where $B'_t > 0$ depends only on $t$, thus there is a $B_t$ such that $\max_j |x_j| \leq B_t/(2n-2)!!$.
\end{proof}
\begin{theorem}\label{thm:nonfat}
Continuing under the assumptions and notation of Theorem~\ref{thm:main}, let $\rho \vdash 2n$ be a non-fat even partition with corresponding eigenvalue $\eta_\rho =  \sum^F_{j=1} x_j P_{\rho,j}$.  Then
$$| \eta_{\rho} | \leq G_t|\zeta|/\sqrt{n} = o(| \zeta |),$$
where $G_t > 0$ depends only on $t$.
\end{theorem}
\begin{proof}
Putting everything together, we have
\begin{align*}
|\eta_\rho| &= \big| \sum_{j=1}^{F_t}x_j P_{\rho, j} \big|\\
&\leq F_t~\max_j |x_j|~\max_j |P_{\rho, j}|\\
&\leq F_t \frac{B_t}{(2n-2)!!}~\max_j |P_{\rho, j}| &\quad(\text{Theorem}~\ref{thm:main})\\
&\leq F_t \frac{B_t}{(2n-2)!!}~\sqrt{\frac{|\mathcal{M}_{2n}| 2^{n+1}(n-1)!}{C_t(2n)^{2(t+1)}}} &\quad(\text{Lemma~\ref{lem:medDims} \& Lemma~\ref{lem:eigsbound}})\\
&\leq F_tB_tD_t \frac{(2n-1)!!}{(2n-2)!!}\frac{|\zeta|}{2n}~~&\text{ where }D_t > 0\text{ depends only on }t\\
&= G_t|\zeta|/\sqrt{n} &\quad(\text{Proposition~\ref{prop:sqrt}})
\end{align*}
where $G_t > 0$ depends only on $t$, as desired.
\end{proof}
The only eigenvalue of $\widetilde{A}(\Gamma_t)$ not accounted for by Theorem~\ref{thm:main} or Theorem~\ref{thm:nonfat} is $\eta_{F_t} = \eta_{(n-t,1^t)}$, which we now handle separately in the theorem below. 
\begin{theorem}\label{thm:fattest}
$\eta_{F_t} = \eta_{(n-t,1^t)} = \zeta$.
\end{theorem}
\begin{proof}
By Theorem~\ref{thm:supp}, we can write the characteristic vector $f$ of a canonically intersecting family as 
\[f = \sum_{i=1}^{|\mathcal{M}_{2n}|} a_iv_i = \sum_{\lambda \vdash n}\sum_{i=1}^{\dim 2\lambda} a_{\lambda,i} v_{\lambda,i} = \sum_{\lambda \text{ fat}}\sum_{i=1}^{\dim 2\lambda} a_{\lambda,i} v_{\lambda,i},\]
where $\{v_{\lambda,i}\}_{i=1}^{\dim 2\lambda}$ is an orthonormal set of eigenvectors of $\widetilde{A}(\Gamma_t)$ that forms a basis for the irreducible $2\lambda$.  Let $w_\lambda = a_{\lambda,1}^2 + a_{\lambda,2}^2 + \cdots + a_{\lambda,\dim 2\lambda}^2$.
By Lemma~\ref{lem:fatnonzero} we have $\hat{f}(2(n-t,1^t)) \neq 0$ for some $f$, which implies that $w_{(n-t,1^t)} \neq 0$. Setting $\alpha = |\mathcal{F}_T|/(2n-1)!!$ and revisiting Equation (\ref{eq:ratio}) gives us
\[ 0 = f^\top\widetilde{A}(\Gamma_t)f = \sum_{\lambda \text{ fat}} \eta_\lambda w_\lambda = \alpha^2 + \zeta(\alpha-\alpha^2 - w_{(n-t,1^t)}) + \eta_{(n-t,1^t)} w_{(n-t,1^t)}. \]
The definition of $\zeta$ implies $\zeta(\alpha-\alpha^2) = -\alpha^2$. Since $w_{(n-t,1^t)} \neq 0$, we have $\eta_{(n-t,1^t)} = \zeta$.
\end{proof}
\begin{proof}[Proof of the First Part of Theorem~\ref{thm:tekr}]
\noindent By Theorem~\ref{thm:main} and Theorem~\ref{thm:fattest}, there exists a pseudo-adjacency matrix of the perfect matching $t$-derangement graph $\widetilde{A}(\Gamma_t)$ with eigenvalues satisfying $\eta_1 = 1$ and $\eta_\lambda = \zeta$ for each non-trivial fat even partition $\lambda \vdash 2n$. By Theorem~\ref{thm:nonfat}, for each non-fat even partition $\rho \vdash 2n$, we have $|\eta_\rho| < |\zeta|$ for $n$ sufficiently large, thus $\zeta$ is the minimum eigenvalue of $\widetilde{A}(\Gamma_t)$ for sufficiently large $n$. By our choice of $\zeta$ and the ratio bound, we have that
\[ |\mathcal{F}| \leq |V|\frac{-\eta_{\min}}{\eta_1 - \eta_{\min}} = (2n-1)!! \frac{-\zeta}{1 - \zeta} = (2(n-t)-1)!!\]
for any $t$-intersecting family $\mathcal{F} \subseteq \mathcal{M}_{2n}$, as desired. 
\end{proof}
\noindent We say that two families $\mathcal{F,G} \subseteq \mathcal{M}_{2n}$ are \emph{$t$-cross-intersecting} if $|m \cap m'| \geq t$ for all $m \in \mathcal{F}, m' \in \mathcal{G}$. 
As a bonus, we have the following cross-intersecting variant of Theorem~\ref{thm:tekr} that follows easily from the \emph{cross-ratio bound} stated below.
\begin{theorem}[Cross-Ratio Bound~\cite{Ellis11}]
Let $\widetilde{A}(\Gamma)$ be a pseudo-adjacency matrix of a graph $\Gamma$ with eigenvalues $|\eta_1| \geq |\eta_2| \geq \cdots \geq |\eta_{n}|$ and corresponding eigenvectors $v_1, v_2 \cdots , v_n$. Let $S,T \subseteq V$ be sets of vertices such that there are no edges between $S$ and $T$. Then 
\[\sqrt{\frac{|S|\cdot |T|}{|V|^2}} \leq \frac{|\eta_2|}{\eta_1 + |\eta_2|}.\] 
If equality holds, then
\[ 1_S,1_T \in \emph{Span}(\{v_1\} \cup \{v_i : |\eta_i| = |\eta_2|\}).\]
\end{theorem}
\begin{theorem}\label{thm:cross}
For any $t \in \mathbb{N}$, if $\mathcal{F,G} \subseteq \mathcal{M}_{2n}$ is $t$-cross-intersecting, then
\[|\mathcal{F}|\cdot |\mathcal{G}|  \leq ((2(n-t) - 1)!!)^2\]
for sufficiently large $n$ depending on $t$. 
\end{theorem}
\begin{proof}
\noindent By Theorem~\ref{thm:main} and Theorem~\ref{thm:fattest}, there exists a pseudo-adjacency matrix of the perfect matching $t$-derangement graph $\widetilde{A}(\Gamma_t)$ with eigenvalues satisfying $\eta_1 = 1$ and $\eta_\lambda = \zeta$ for each non-trivial fat even partition $\lambda \vdash 2n$. By Theorem~\ref{thm:nonfat}, for each non-fat even partition $\rho \vdash 2n$, we have $|\eta_\rho| < |\zeta|$ for $n$ sufficiently large, thus $\zeta$ is the second-largest eigenvalue in absolute value for sufficiently large $n$. By our choice of $\zeta$ and the cross-ratio bound, we have that 
$$|\mathcal{F}|\cdot |\mathcal{G}| \leq ((2(n - t)-1)!!)^2,$$
for any $t$-cross-intersecting $\mathcal{F},\mathcal{G} \subseteq \mathcal{M}_{2n}$.
\end{proof}
\medskip
\noindent To finish the proof of our main result, we must show that the largest $t$-intersecting families are precisely the canonically $t$-intersecting families for sufficiently large $n$. We shall do this, albeit a bit indirectly, by proving a stability theorem for $t$-intersecting families of $\mathcal{M}_{2n}$.

\section{Stability Preliminaries and a Proof Sketch}\label{sec:sketch2}
Our next result is the following stability theorem for $t$-intersecting families of $\mathcal{M}_{2n}$.
\begin{theorem}\label{thm:stability}
For any $\epsilon \in (0,1/\sqrt{e})$, $n > n(\epsilon)$, any $t$-intersecting family of $\mathcal{M}_{2n}$ larger than $(1 - 1/\sqrt{e} + \epsilon)(2(n - t)-1)!!$ is contained in a canonically $t$-intersecting family.
\end{theorem}
\noindent It is clear that this theorem implies the characterization of the extremal families stated in Theorem~\ref{thm:tekr}. In a recent note~\cite{Lindzey18}, the $t=1$ case of Theorem~\ref{thm:stability} was proven.  The reader may find it useful to first peruse the proof of~\cite[Theorem 2]{Lindzey18} since the $t \geq 2$ case is along the same lines as the $t=1$ case, only the latter is much simpler.

The lion's share of the remainder will be spent towards proving a key lemma, which roughly asserts that almost all of the members of a large $t$-intersecting family have a set of $t$ disjoint edges in common.  For any set of disjoint edges $S \subseteq E(K_{2n})$, we define the set $\mathcal{F} \! \downarrow_S := \{ m \in \mathcal{F} : S \subseteq m\}$ to be the \emph{restriction} of $\mathcal{F} \subseteq \mathcal{M}_{2n}$ with respect to $S$.

\begin{lemma}[Key Lemma]\label{lem:key}
Let $c \in (0,1)$. For any $t \in \mathbb{N}$, if $\mathcal{F} \subseteq \mathcal{M}_{2n}$ is a $t$-intersecting family such that $|\mathcal{F}| \geq c(2(n-t)-1)!!$, then there is a set of $t$ disjoint edges $T \subseteq E(K_{2n})$ such that 
$$|\mathcal{F} \setminus  \mathcal{F}\!\downarrow_{T} | = O((2(n-t-1)-1)!!)$$
for sufficiently large $n$ depending on $c$ and $t$.
\end{lemma}
\noindent For any set of disjoint edges $m$ and vertex $u \in V(m)$, let $m(u)$ be the other endpoint of the edge incident to $u$ in $m$, that is, the \emph{partner} of $u$. We now show that the key lemma implies the stability theorem.	
\begin{proof}[Proof of Theorem~\ref{thm:stability}]
Let $\mathcal{F}$ be a $t$-intersecting family of size $c(2(n-t)-1)!!$ such that $c \in (1 - 1/\sqrt{e}, 1)$. Assuming the key lemma, there is a set of $t$ disjoint edges $T \subseteq E(K_{2n})$ such that $|\mathcal{F} \setminus \mathcal{F}\!\downarrow_{T} | = O((2(n-t-1)-1)!!)$, which implies that
\begin{align}\label{eq:foo}
|\mathcal{F} \!\downarrow_{T} | \geq (c - O(1/n))(2(n-t)-1)!!.
\end{align}
For a contradiction, suppose there exists a perfect matching $m \in \mathcal{F}$ such that $T \not\subseteq m$, and let $s = |m \cap T|$. Since $\mathcal{F}$ is $t$-intersecting, any member of $\mathcal{F} \!\downarrow_T$ shares $t$ edges with $m$, and therefore no member of $\mathcal{F} \!\downarrow_T$ can be a $(t-s)$-derangement with respect to $m$ when restricted to the complete subgraph $K_{2n} \setminus V(T)$. This implies that
\[  |\mathcal{F} \!\downarrow_{T}| \leq (2(n-t)-1)!! - D_2(n-t,t-s) \leq (2(n-t)-1)!! - D_2(n-t,1). \]
Assume $t-s=1$, and let $ij$ be the edge of $T$ that is not contained in $m$. Then $m(i),m(j) \notin V(T)$, and every member of $\mathcal{F} \!\downarrow_T$ that contains $\{m(i),m(j)\}$ that is also a derangement with respect to $m$ when restricted to $K_{2n} \setminus V(T)$ cannot $t$-intersect $m$. This gives us
\begin{align*}
|\mathcal{F} \!\downarrow_{T}| &\leq (2(n-t)-1)!! - D_2(n-t,t-s)\\
&\leq (2(n-t)-1)!! - D_2(n-t,1) - D_2(n-t-1,1)\\
&= (1 - 1/\sqrt{e} - o(1))(2(n-t)-1)!!,  
\end{align*}
where the equality follows from Equation (\ref{eq:derange}). But this contradicts (\ref{eq:foo}) for $n$ sufficiently large depending on $c$ and $t$.
\end{proof}

\noindent By the argument above, it suffices to prove the key lemma. The core of its proof is a generalization of the ratio bound.  Recall that $U_t$ is space of real-valued functions on perfect matchings that are supported on the fat even irreducibles as defined in Section~\ref{sec:low}.

\begin{theorem}[Stability Version of Ratio Bound~\cite{Ellis11}]\label{thm:stableratio}
Let $\widetilde{A}(\Gamma)$ be the pseudo-adjacency matrix of a regular graph $\Gamma = (V,E)$ with eigenvalues $\eta_{\min} \leq \cdots \leq \eta_2 \leq   \eta_1$ and corresponding orthonormal eigenvectors $v_{\min}, \cdots, v_2,v_{1}$. Let $\mu = \min_i \{ \eta_i : \eta_i \neq \eta_{\min} \}$.
Let $S \subseteq V$ be an independent set of vertices of measure $\alpha = |S|/|V|$. 
Let $D = \| P_{U_t^\perp}(1_S) \|$ be the Euclidean distance from the characteristic function of $S$ to the subspace $U_t$. Then
\[ D^2 \leq \alpha \frac{(1-\alpha)|\eta_{\min}| - |\eta_{1}|\alpha}{|\eta_{\min}| - |\mu|}.\]
\end{theorem}

\noindent For any subset of vertices $V' \subseteq V(K_{2n})$, let $\Delta(V') \subseteq E(K_{2n})$ be the set of edges that have at least one endpoint in $V'$. The next lemmas follow from Theorem~\ref{thm:cross} and will be needed for a couple of combinatorial arguments.  Let $T^* := \{\{1,2\},\{3,4\},\cdots,\{2t-1,2t\}\}$.

\begin{lemma}\label{lem:smallcross}
Let $\mathcal{F} \subseteq \mathcal{M}_{2n}$ be a $t$-intersecting family.  For any set of $t$ disjoint edges $T \subseteq E(K_{2n})$ such that $T \cap T^* = \emptyset$, $|V(T) \cap \{2i-1,2i\}| \geq 1$ for all $i \in [t]$, and $T \subseteq \Delta([2t])$, we have
\[ |\mathcal{F} \!\downarrow_{T^*}| \cdot | \mathcal{F} \!\downarrow_{T}| \leq ((2(n-2t)-1)!!)^2.\]
\end{lemma}
\begin{proof}
Note that because $\mathcal{F}$ is a $t$-intersecting family, we have that $\mathcal{F} \!\downarrow_{T^*}$ and $\mathcal{F} \!\downarrow_{T}$ are cross-$t$-intersecting on edges of $K_{2n} \setminus \{ [2t] \cup V(T)\}$.  
By our choice of $T$, we have 
$$ V(T) \setminus [2t] = \{u_1,u_2,\cdots,u_k\} \quad \text{ and } \quad   [2t] \setminus V(T) = \{v_1,v_2,\cdots,v_k\}$$ 
for some $k \leq t$.  Define the involution $\pi := (u_1~v_1) (u_2~v_2) \cdots (u_k~v_k)$. For any two perfect matchings $m \in \mathcal{F} \!\downarrow_{T^*}$ and $m' \in \pi \left( \mathcal{F} \!\downarrow_{T} \right)$, every edge of $m$ and $m'$ has either both of its endpoints in $[2t]$ or none of its endpoints in $[2t]$. Since $\pi$ fixes every $v \notin [2t] \cup V(T)$, we also have that $m,m'$ are cross-$t$-intersecting on edges of $K_{2n} \setminus \{ [2t] \cup V(T)\}$.  By deleting $[2t]$ we obtain $\left(\mathcal{F} \!\downarrow_{T^*}\right)'$ and $\left(\pi \left( \mathcal{F} \!\downarrow_{T} \right)\right)'$ which are $t$-cross-intersecting families of $\mathcal{M}_{2(n-t)}$. By Theorem~\ref{thm:cross}, we deduce that
\begin{align*}
|\mathcal{F} \!\downarrow_{T^*}| \cdot | \mathcal{F} \!\downarrow_{T} | = |\mathcal{F} \!\downarrow_{T^*}| \cdot | \pi \left( \mathcal{F} \!\downarrow_{T} \right)| &= |\left(\mathcal{F} \!\downarrow_{T^*}\right)'| \cdot |\left(\pi \left( \mathcal{F} \!\downarrow_{T} \right)\right)'|\\
 &\leq  ((2(n-2t)-1)!!)^2,
\end{align*}
as desired.
\end{proof}
\noindent A similar argument can be used to show the following.
\begin{lemma}\label{lem:1cross}
Let $\mathcal{F} \subseteq \mathcal{M}_{2n}$ be a $t$-intersecting family.  For any $i,j,k \in V(K_{2n})$, we have
\[ |\mathcal{F} \!\downarrow_{ij}| \cdot | \mathcal{F} \!\downarrow_{ik}| \leq ((2(n-t-1)-1)!!)^2.\]
\end{lemma}

The \emph{perfect matching transposition graph} is the graph $\Upsilon_n$ defined such that two perfect matchings $m,m' \in \mathcal{M}_{2n}$ are adjacent if they differ by a \emph{partner swap}, that is, there exists a transposition $\tau \in S_{2n}$ such that $\tau m = m'$. In other words, the adjacency matrix of $\Upsilon_n$ is $A_{(2,1^{n-2})} \in \mathcal{A}$. For any graph $\Gamma = (V,E)$, the \emph{h-neighborhood} of a set $X \subseteq V$ is the set of vertices $N_h(X) := \{ v \in V : \text{dist} (v,X) \leq h\}$ where $\text{dist} (v,X)$ is the length of a shortest path from $v$ to any vertex of $X$. 

It is instructive to think of $h$-neighborhoods in $\Upsilon$ as balls of radius $h$ in a discrete metric space, as perfect matchings in a ball of small radius around some perfect matching in $\Upsilon$ are all structurally quite similar, i.e., they share many edges. The following isoperimetric inequality plays a pivotal role in our proof of the key lemma.
\begin{proposition}\emph{\cite{Lindzey18}}\label{prop:iso}
Let $X \subset V(\Upsilon_n)$ such that $|X| \geq a(2n-1)!!$ for some $a \in (0,1)$. Then for any $h \in \mathbb{N}$ such that 
$$h > h_0 = \sqrt{\frac{n}{2} \ln (1/a) },$$
the following holds: 
\[ N_h(X) \geq \left(1 - \exp \left(\frac{-2(h-h_0)^2}{n} \right) \right)(2n-1)!!.\]
\end{proposition}

\noindent We are now in a position to sketch a proof of the key lemma.

\subsection*{Proof Sketch II}
Our starting point is the fact shown in Section~\ref{sec:pseudo}, that for sufficiently large $n$ the eigenvalues $\{\eta_\lambda\}_{\lambda \vdash n}$ of $\widetilde{A}(\Gamma_t)$ satisfy
\[ | \eta_\mu | = o(|\zeta|) \quad \text{ for all non-fat }\mu \vdash n.\] 
With this in hand, we use the stability version of the ratio bound to show the characteristic function $f$ of any large $t$-intersecting family $\mathcal{F}$ is close in Euclidean distance to $U_t$. 

We then use the fact that $f$ is close to $U_t$ along with Proposition~\ref{prop:iso} to exhibit two perfect matchings $m_1 \in \mathcal{F}$ and $m_0 \notin \mathcal{F}$ that are structurally quite similar, i.e., sharing many edges, yet their orthogonal projections $[P_{U_t}f](m_1) \text{ and } [P_{U_t}f](m_0)$ onto the subspace $U_t$ are close to 1 and 0 respectively.  The projector $P_{U_t}$ is the sum of primitive idempotents $E_\lambda$ of the perfect matching association scheme such that $\lambda$ is fat:
\[  P_{U_t} = E_{(n)} + E_{(n-1,1)} + \cdots + E_{(n-t,1^t)}.\]
This projector is quite difficult to work with in practice due to the fact that the character theory of $\mathcal{M}_{2n}$ is considerably less understood than the classical character theory of $S_n$. For example, no Murnaghan-Nakayama-type rule or Jacobi-Trudi-type determinantal identity is known for expressing the zonal characters, which are outstanding open questions in the theory of zonal and Jack symmetric functions~\cite{Stanley89}. 
In light of this, we must use some barebone combinatorial and representation-theoretical arguments to find crude estimates of these projections, which we cover in the next section.  

Once we have boiled these projections down to their combinatorial essence, we use the fact that $m_1$ and $m_0$ share many edges to prove that $\mathcal{F}$ has a large restriction with respect to some set of $t$ disjoint edges $T$, however, not large enough to deduce the key lemma. Following a bootstrapping argument of Ellis~\cite{Ellis11}, we use our bounds on $t$-cross-intersecting families to show that almost every member of $\mathcal{F}$ has an edge in common with $T$. This fact, after an induction on $t$, leads to a proof of the key lemma. 

The asymptotics of permutations and perfect matchings bear a strong resemblance, so there are points in the proof that follow from Ellis~\cite{Ellis11} \emph{mutatis mutandis}. We have adopted a notation that is consistent with Ellis' at these places.

\section{Polytabloids and the Characters of $\mathcal{M}_{2n}$}\label{sec:combchar}

Although the character theory of $\mathcal{M}_{2n}$ is determined by the zonal polynomials expressed in the power-sum basis, arriving at tractable combinatorial expressions for these coefficients is considerably more difficult than it is for Schur functions. Instead, we work with combinatorial formulas for these quantities that stem from the fact that spherical functions of the Gelfand pair $(S_{2n},S_2 \wr S_n)$ are the projections of characters of $S_{2n}$ onto the space $H_n$-invariant functions. For a more detailed account, see~\cite[Ch. 11]{CST}.

For any $2\lambda$-tableau $T$, let $\text{row}_T(i)$ be the index of the row of $T$ that the cell labeled $i$ belongs to, and let $\text{col}_T(i)$ be the index of the column of $T$ that the cell labeled $i$ belongs to.  
We say that a $2\lambda$-tabloid $\{T\}$ \emph{covers} $m \in \mathcal{M}_{2n}$ if $\text{row}_T(i) = \text{row}_T(m(i))$ for all $i \in [2n]$. A $2\lambda$-tableau $T$ is \emph{m-aligned} with respect to a perfect matching $m \in \mathcal{M}_{2n}$ if $\{T\}$ covers $m$ and for any $i \in [2n]$ we have $\{\text{col}_T(i),\text{col}_T(m(i))\} = \{2j-1,2j\}$ for some $j \in [n]$.

For example, the perfect matching
$$1~7|2~4|3~8|5~12|6~11|9~10|13~14 \quad \text{ is not covered by } \quad  \ytableausetup{mathmode,boxsize=1.2em,tabloids = true}
\begin{ytableau}
1 & 2 & 3 & 4 & 5 & 6 & 11 & 12 \\
7 & 8 & 9 & 10\\ 
13 & 14  \\ 
\end{ytableau}~,$$ 
but the perfect matching $1~12|2~3|4~5|6~11|7~9|8~10|13~14$ is covered by this tabloid, as illustrated below:
\[ 
\begin{tikzpicture}
    \matrix (m) [
                matrix of math nodes, 
                nodes in empty cells,
                minimum width=width("8"),
            ] {
                1 & 2 & 3 & 4 & 5 & 6 & 11 & 12 \\
                7 & 8 & 9 & 10 \\
                13 & 14 \\
            };
            \draw (m-1-1.center) to [bend left = 20]  (m-1-8.center);
            \draw (m-1-2.center) --  (m-1-3.center);
             \draw (m-1-4.center) -- (m-1-5.center);
            \draw (m-1-6.center) -- (m-1-7.center);

             \draw (m-2-1.center) to [bend left = 40]  (m-2-3.center);
             \draw (m-2-2.center) to [bend right = 40]  (m-2-4.center);

             \draw (m-3-1.center) -- (m-3-2.center);          
\end{tikzpicture}.
\]
Also note that the tableau shown above is a $m^*$-aligned.  
Let
$$ \mathcal{P} := \frac{1}{|H_n|} \sum_{h \in H_n} e_h$$
denote the projection from $\mathbb{R}[S_{2n}/H_n]$ to its bi-$H_n$-invariant subalgebra $\mathbb{R}[H_n \backslash S_{2n} / H_n]$.  
Let $1_{\{T\}} \in \mathbb{R}[\mathcal{M}_{2n}]$ be the characteristic function of the set of perfect matchings that are covered by $\{T\}$, that is, 
\[
1_{\{T\}}(m) = 
\begin{cases}
1 &\text{ if } \{T\} \text{ covers }m,\\
0 &\text{ otherwise}
\end{cases}.
\]
for all $m \in \mathcal{M}_{2n}$. For any $\lambda \vdash n$, define the map
$$\mathcal{I}_\lambda' : \{ e_{\{T\}} \in \mathbb{R}[S_{2n}] : \{T\} \text{ is a } 2\lambda\text{-tabloid}\} \rightarrow \mathbb{R}[\mathcal{M}_{2n}]  \text{~~~such that~~~} \mathcal{I}_\lambda'(e_{\{T\}}) = 1_{\{T\}}.$$
Recall from Section~\ref{sec:sym} that
$$\text{Span}\{ e_{\{T\}} : T \text{ is a } 2\lambda\text{-tabloid}\} \cong \mathbb{R}[\mathbb{T}_{2\lambda}].$$
Let $\mathcal{I}_\lambda$ be the linear extension of $\mathcal{I}_\lambda'$, that is,
$$\mathcal{I}_\lambda : \mathbb{R}[\mathbb{T}_{2\lambda}] \rightarrow \mathbb{R}[\mathcal{M}_{2n}].$$
\noindent For any $2\lambda$-tabloid $\{T\} \in \mathbb{T}_{2\lambda}$ and $\sigma \in S_{2n}$, we have that $m \in \mathcal{M}_{2n}$ is covered by $\{T\}$ if and only if $\sigma m$ is covered by $\sigma\{T\}$; therefore, $\mathcal{I}_\lambda$ \emph{intertwines} the permutation representations $\mathbb{R}[\mathbb{T}_{2\lambda}]$ and $\mathbb{R}[\mathcal{M}_{2n}]$, in symbols:
$$ \mathcal{I}_\lambda(\sigma e_{\{T\}}) = \sigma \mathcal{I}_\lambda(e_{\{T\}}) \quad \text{ for all } \sigma \in S_{2n}.$$ 
By Schur's lemma, this implies for each irreducible $2\lambda$ that $\mathcal{I}_\lambda$ acts either as an isomorphism or as the zero map, and it is simple to show that the latter is not the case.
In particular, this shows for any standard Young tableau $T$ of shape $2\lambda$ and corresponding polytabloid $e_T$ that 
$$f_T := \mathcal{I}_\lambda e_T =  \sum_{\sigma \in C_T} \text{sign}(\sigma) \mathcal{I}_\lambda e_{\{\sigma T\}} = \sum_{\sigma \in C_T} \text{sign}(\sigma) 1_{\{\sigma T\}}$$ 
lies in the $2\lambda$ irreducible subspace of $\mathbb{R}[\mathcal{M}_{2n}]$, and moreover, that
$$ \{ f_T \in \mathbb{R}[\mathcal{M}_{2n}] : T \text{ is a standard Young tableau of shape }2\lambda  \}$$ 
is a basis for the $2\lambda$ irreducible subspace of $\mathbb{R}[\mathcal{M}_{2n}]$. 

For any $\lambda \vdash n$, the spherical function $\omega^\lambda \in \mathbb{R}[\mathcal{M}_{2n}]$ is the unique $H_n$-invariant function $\omega^\lambda \in \mathcal{I}_\lambda 2\lambda \leq \mathbb{R}[\mathcal{M}_{2n}]$ (equivalently, bi-$H_n$-invariant function of $2\lambda \leq \mathbb{R}[S_{2n}]$) which satisfies  $\omega^\lambda(m^*) = \omega^\lambda_{(1^n)} = 1$.  For any $\pi \in S_{2n}$ and $m = \pi m^*$, let $\omega^\lambda_{d(m,\cdot)}$ denote the zonal spherical function translated by $\pi$ so that $ \omega^\lambda_{d(m,\cdot)}(m) = \omega^\lambda_{(1^n)} = 1$.

For any tableau $T$ of even shape $2\lambda$, let $C'_T \leq C_T$ be the subgroup of the column-stabilizer of $T$ that stabilizes the odd-indexed columns of $T$ and acts trivially on the even-indexed columns.  More precisely, if we have $C_T \cong S_{(2\lambda')_1} \times S_{(2\lambda')_2} \times \cdots \times S_{(2\lambda')_{2\lambda_1}}$, then $C_T' \cong S_{(\lambda')_1} \times S_{(\lambda')_2} \times \cdots \times S_{(\lambda')_{\lambda_1}}$.  For any $m \in \mathcal{M}_{2n}$ and $m$-aligned $2\lambda$-tableau $T$, we have
\begin{align}
\omega^\lambda_{d(m,\cdot)} &= \frac{1}{\lambda'_1! \cdots \lambda'_{\lambda_1}!} \mathcal{P}\mathcal{I}_\lambda \sum_{\sigma \in C_T} \text{sign}(\sigma)  e_{\{\sigma T\}}\label{eq:P}\\
&= \frac{1}{\lambda'_1! \cdots \lambda'_{\lambda_1}! ~ |H|}\sum_{h \in H}\sum_{\sigma \in C_T} \text{sign}(\sigma) \mathcal{I}_\lambda e_{h\{\sigma T\}}\\
&= \frac{1}{|H|}\sum_{h \in H}\sum_{\sigma \in C_T'} \text{sign}(\sigma) 1_{\{h\sigma T\}}\label{eq:orb},
\end{align}
\noindent where $H$ is a translate of $H_n$~(see~\cite[Ch. 11]{CST}). The expression above together with the projection formula below gives us an explicit but admittedly complicated combinatorial formula for computing these projections.
\begin{lemma}\label{lem:fatproj}
\emph{\cite{Lindzey17}}
Let $E_{\mu} : \mathbb{R}[\mathcal{M}_{2n}] \rightarrow 2\mu$ denote the orthogonal projection onto $2\mu$.  For any $f \in \mathbb{R}[\mathcal{M}_{2n}]$, we have 
\[ [E_{\mu} f](m) = \frac{\dim {2\mu}}{(2n-1)!!} \sum_{\lambda \vdash n} \left( \sum_{m' : d(m,m') = \lambda} f(m')  \right) \omega^\mu_\lambda,\]
equivalently,
\[ [E_{\mu} f](m) = \frac{\dim {2\mu}}{(2n-1)!!} \sum_{m' \in \mathcal{M}_{2n}} f(m')~\omega^\mu_{d(m,m')}.\]
\end{lemma}

\noindent Without further ado, we begin the proof of the key lemma.

\section{Proof of the Key Lemma}\label{sec:stab}
Let $c \in (0,1)$ and let $\mathcal{F}$ be a $t$-intersecting family of perfect matchings such that 
$$|\mathcal{F}| \geq c(2(n-t)-1)!!.$$
Recall that our goal is to show there exists a set of $t$ disjoint edges $T \subseteq E(K_{2n})$ such that $|\mathcal{F} \setminus  \mathcal{F}\!\downarrow_{T} | = O((2(n-t-1)-1)!!)$ for sufficiently large $n$ depending on $c$ and $t$.\\

\noindent Let $f$ be the characteristic function of $\mathcal{F}$, $\alpha = |\mathcal{F}|/(2n-1)!! \geq c/(\!(2n-1)\!)_t$, and $D$ be the Euclidean distance from $f$ to $U$. Setting $S = \mathcal{F}$ and applying our pseudo-adjacency matrix $\widetilde{A}(\Gamma_t)$ to Theorem~\ref{thm:stableratio} gives us
\begin{align*}
D^2 \leq \alpha \frac{(1-\alpha)|\zeta| - \alpha}{|\zeta| - |\mu|} &= \alpha \frac{(1-\alpha) - \alpha/|\zeta|}{1 - |\mu|/|\zeta|}\\
&\leq \alpha \frac{(1-\alpha) - c}{1 - O(1/\sqrt{n})}\quad \quad (\text{Theorem~\ref{thm:nonfat}})\\
&\leq \alpha \frac{1 - c}{1 - O(1/\sqrt{n})} \\
&\leq \delta(1 + O(1/\sqrt{n})) \frac{|\mathcal{F}|}{(2n-1)!!}
\end{align*}
where $\delta := 1-c$. We have 
$$\| P_{U_t^\perp} f \|_2^2 = \| f - P_{U_t}f\|_2^2 = D^2 \leq  \delta(1 + O(1/\sqrt{n})) \frac{|\mathcal{F}|}{(2n-1)!!},$$
which tends to zero as $n \rightarrow \infty$. This already shows that $f$ is close to $U_t$, but we now seek a combinatorial explanation for this proximity.

By Lemma~\ref{lem:fatproj}, we may write
$$P_m := [P_{U_t} f](m)  =  \sum_{\mu~\text{fat}} \frac{\dim {2\mu} }{(2n-1)!!} \sum_{m' \in \mathcal{M}_{2n}} f(m')~\omega^\mu_{d(m,m')}.$$
Now we have
\[ D^2 = \frac{1}{(2n-1)!!} \left( \sum_{m \in \mathcal{F}} (1-P_m)^2 + \sum_{m \notin \mathcal{F}} P_m^2 \right) \leq \delta(1 + O(1/\sqrt{n})) \frac{|\mathcal{F}|}{(2n-1)!!}, \]
which gives us,
\begin{align}\label{eq:projs}
\sum_{m \in \mathcal{F}} (1-P_m)^2 + \sum_{m \notin \mathcal{F}} P_m^2  \leq \delta(1 + O(1/\sqrt{n}))|\mathcal{F}|.
\end{align}
Following Ellis, pick $C > 0$ large enough so that 
\[\sum_{m \in \mathcal{F}} (1-P_m)^2 + \sum_{m \notin \mathcal{F}} P_m^2  \leq \delta(1 + O(1/\sqrt{n}))|\mathcal{F}| \leq |\mathcal{F}|(1-1/\sqrt{n}) \delta(1 + C/\sqrt{n}).\]
By the nonnegativity of each term on the left-hand side of (\ref{eq:projs}), at least $|\mathcal{F}|/\sqrt{n}$ members of $\mathcal{F}$ satisfy $(1 - P_m)^2 < \delta(1 + C/\sqrt{n})$; therefore, there exists a set
\[\mathcal{F}_1 := \{m \in \mathcal{F} : (1-P_m)^2 < \delta (1+C/\sqrt{n})\}\]
such that $|\mathcal{F}_1| \geq |\mathcal{F}|/\sqrt{n}$. The inequality~(\ref{eq:projs}) also implies that $P^2_m < 2\delta/n$ for every $m \notin \mathcal{F}$ with the exception of at most $n|\mathcal{F}|(1+O(1/\sqrt{n}))/2$ non-members, thus there is a set
$$\mathcal{F}_0 := \{ m \not \in \mathcal{F} : P_m^2 < 2\delta/n\} $$
such that
$$|\mathcal{F}_0| \geq (2n-1)!! - c(2(n-t)-1)!! - cn(2(n-t)-1)!!(1+O(1/\sqrt{n})) /2.$$
Observe that the projections of the elements of $\mathcal{F}_0$ and $\mathcal{F}_1$ are close to 0 and 1 respectively. \\

\noindent We now show there exist $m_1 \in \mathcal{F}_1$ and $m_0 \in \mathcal{F}_0$ that are close together in the graph $\Upsilon$, differing by $O(\sqrt{n \log n}$) partner swaps, which implies that $m_1$ and $m_0$ share many edges.

In particular, we claim there is a path $m_1p_2p_3 \cdots p_{\ell-1}m_0 $ in $\Upsilon$ of length at most $2\sqrt{(t+2) \frac{n}{2} \ln(2n)}$. Take $a = 1/(2n)^{t+2}$ and $h = 2h_0$ in Proposition~\ref{prop:iso}.  Since 
$$|\mathcal{F}_1| \geq c(2(n-t)-1)!!/n \geq (2n-1)!!/(2n)^{t+2}$$ 
for sufficiently large $n$, Proposition~\ref{prop:iso} gives us
$$ |N_h(\mathcal{F}_1)| \geq \left(1 - \frac{1}{n^{t+2}} \right)(2n-1)!!.$$
Since $|\mathcal{F}_0| > (2n-1)!!/(2n)^{t+2}$ for sufficiently large $n$, we have $|\mathcal{F}_0 \cap N_h(\mathcal{F}_1)| \neq \emptyset$, thus there is a path from $m_1$ to $m_0$ in $\Upsilon$ of length no more than $2\sqrt{(t+2) \frac{n}{2} \ln(2n)} = O(\sqrt{n \log n})$.\\

\noindent The inequality~(\ref{eq:projs}) and the foregoing shows that
\[ 1 - \sqrt{\delta(1 + C/\sqrt{n})} < P_{m_1} \text{ and } P_{m_0} < \sqrt{2\delta/\sqrt{n}}.\]
Combining these inequalities reveals that
$$P_{m_1} - P_{m_0} \geq (1- \sqrt{\delta} - O(1/\sqrt[4]{n})).$$
Rewriting using Lemma~\ref{lem:fatproj} gives us
$$\sum_{\mu~\text{fat}} \frac{\dim {2\mu}}{(2n-1)!!} \left( \sum_{m \in \mathcal{F}} \omega^\mu_{d(m_1,m)} -  \sum_{m \in \mathcal{F}} \omega^\mu_{d(m_0, m)} \right) \geq (1- \sqrt{\delta} - O(1/\sqrt[4]{n})).$$
By averaging, there exists a fat $\mu \neq (n)$ such that
$$ \frac{\dim {2\mu}}{(2n-1)!!} \left( \sum_{m \in \mathcal{F}} \omega^\mu_{d(m_1,m)} -  \sum_{m \in \mathcal{F}} \omega^\mu_{d(m_0, m)} \right) \geq \frac{(1- \sqrt{\delta} - O(1/\sqrt[4]{n}))}{F_t}.$$
Rearranging gives us
$$  \sum_{m \in \mathcal{F}} \omega^\mu_{d(m_1,m)} -  \sum_{m \in \mathcal{F}} \omega^\mu_{d(m_0, m)} \geq \frac{(1- \sqrt{\delta} - O(1/\sqrt[4]{n})) (2n-1)!!}{F_t~\dim 2\mu}.$$
Without loss of generality, we may assume that $m_1 = m^*$ and $m_0 = \pi m^*$ such that $\pi \in S_{2n}$ is a product of $O(\sqrt{n \log n})$ transpositions. By (\ref{eq:P}) and interchanging summations, we have
$$ \sum_{\sigma \in C_T'} \sum_{m \in \mathcal{F}} \text{sign}(\sigma) \left(\mathcal{P}1_{\{\sigma T\}}(m) - \mathcal{P}1_{\pi\{ \sigma T\}}(m) \right) \geq \frac{(1- \sqrt{\delta} - O(1/\sqrt[4]{n}))(2n-1)!!}{F_t~\dim 2\mu}$$
where $T$ is a $2\mu$-tableau that is $m^*$-aligned. By averaging, there exists a $\sigma \in C_T'$ such that
$$\text{sign}(\sigma)   \left( \sum_{m \in \mathcal{F}} \mathcal{P}1_{\{\sigma T\}}(m) -  \sum_{m \in \mathcal{F}}\mathcal{P}1_{\pi\{ \sigma T\}}(m) \right) \geq \frac{(1- \sqrt{\delta} - O(1/\sqrt[4]{n}))(2n-1)!!}{~F_t~|C_T'|~\dim 2\mu}.$$
Without loss of generality, we may assume
$$ \left( \sum_{m \in \mathcal{F}} \mathcal{P}1_{\{\sigma T\}}(m) -  \sum_{m \in \mathcal{F}}\mathcal{P}1_{\pi \{ \sigma T\}}(m)\right) \geq \frac{(1- \sqrt{\delta} - O(1/\sqrt[4]{n}))(2n-1)!!}{F_t~|C_T'|~\dim 2\mu}.$$
By (\ref{eq:orb}), we have
$$   \sum_{h \in H_n} \sum_{m \in \mathcal{F}} \left({1}_{\{h \sigma T\}} - {1}_{\pi \{ h \sigma T\}}\right)(m) \geq \frac{(1- \sqrt{\delta} - O(1/\sqrt[4]{n}))(2n-1)!!~|H_n |}{F_t~|C_T'|~\dim 2\mu}.$$
Note that if $\{h \sigma T\} = \pi \{h \sigma T\}$, then ${1}_{\{h\sigma T\}}(m) - {1}_{\pi \{h\sigma T\}}(m) = 0$.  Let 
$$I := \{ i \in V(K_{2n}) :  \pi(i) \neq i\}$$ 
be the set of vertices moved by $\pi$, and for any tabloid $\{T\}$, let $\{\overline{T}\}$ be the sub-tabloid obtained by deleting the first row of $\{T\}$. After canceling some terms, we have
$$ \sum_{\substack{h \in H_n \\  \exists i \in I : i \in \{\overline{h \sigma T}\}}} \!\! \sum_{m \in \mathcal{F}}  \left({1}_{\{h\sigma T\}} - {1}_{ \pi\{  h \sigma T\}}\right)(m) \geq \frac{(1- \sqrt{\delta} - O(1/\sqrt[4]{n}))(2n-1)!!~|H_n |}{F_t~|C_T'|~\dim 2\mu}.$$
Because $m_1$ and $m_0$ differ by only $O( \sqrt{n \log n})$ partner swaps, we have $|I| = o(n)$.   The number of permutations $h \in H_n$ that send a vertex $i \in I$ to a row of $\{\overline{\sigma T}\}$ is $o(|H_n|)$. Since there are $o(|H_n |)$ terms in the outer summation, by averaging, there is a tabloid $\{h\sigma T\} =: \{S\}$ such that
$$  \sum_{m \in \mathcal{F}}  {1}_{\{S\}}(m) - {1}_{ \pi \{S \}}(m)  \geq \frac{(1- \sqrt{\delta} - O(1/\sqrt[4]{n}))(2n-1)!!~\omega(1)}{ F_t~|C_T'|~\dim 2\mu  }.$$
After absorbing constants depending on $c$ and $t$ and dropping negative terms, we have
$$  \sum_{m \in \mathcal{F}} {1}_{\{S\}}(m)  \geq \frac{(2n-1)!!~\omega(1)}{\dim 2\mu  }.$$
Henceforth, we absorb constant factors on the right-hand side into $\omega(1)$. By the pigeonhole principle, there are $s := n - \mu_1 \leq t$ disjoint edges $S'$ covered by $\{ \overline{S} \}$ such that
$$  | \mathcal{F} \!\downarrow_{S'} \! |  \geq \frac{(2n-1)!!~\omega(1)}{ \dim 2\mu  }.$$
Similarly, there are $t-s$ disjoint edges $S''$ disjoint from $S'$ such that
$$  | \mathcal{F} \!\downarrow_{S' \cup S''} \!|   \geq \frac{(2n-1)!!~\omega(1)}{\dim 2\mu  ~ (2n)_{2(t-s)} }.$$
By Theorem~\ref{thm:dimbound}, we have
$$  | \mathcal{F} \!\downarrow_{S' \cup S''} \!|   \geq \omega((2(n-2t)-1)!!).$$
By relabeling the vertices of $K_{2n}$, we may assume without loss of generality that
$$  | \mathcal{F} \!\downarrow_{T^*} \!|   \geq \omega((2(n-2t)-1)!!).$$
Let $\mathcal{B}$ be the collection of all partitions of $[2t]$ into two parts $A = \{a_1,\cdots,a_t\}$ and $B = \{b_1,\cdots,b_t\}$. Crudely, the set of members of $\mathcal{F}$ with no edge in $T^*$ can be written as
\begin{align}\label{eq:terms}
\mathcal{F} \setminus \bigcup_{i=1}^t \mathcal{F} \!\downarrow_{\{2i-1,2i\}}~=  \bigcup_{\substack{ (A,B) \in \mathcal{B}  \\ a_{i_1}v_{i_1}, \cdots, a_{i_t}v_{i_t}~:~a_{i_j} \neq a_{i_k},~v_{i_j} \neq v_{i_k} \forall j,k \in [t] \\ v_{i_j} \notin A,~a_{i_j}v_{i_j} \notin T^*~\forall j \in [t] } } \mathcal{F} \!\downarrow_{\{a_{i_1}v_{i_1}, \cdots, a_{i_t}v_{i_t}\}}.
\end{align}
By Lemma~\ref{lem:smallcross}, we have
\begin{align}\label{eq:cross}
|\mathcal{F} \!\downarrow_{T^*}\!| \cdot |  \mathcal{F} \!\downarrow_{\{a_{i_1}v_{i_1}, \cdots, a_{i_t}v_{i_t}\}}\!| \leq ((2(n-2t)-1)!!)^2
\end{align} 
for each term on the right-hand side of (\ref{eq:terms}). Since $| \mathcal{F} \!\downarrow_{T^*}\!\! | \geq \omega((2(n-2t)-1)!!)$, the bound (\ref{eq:cross}) implies that 
$$|\mathcal{F} \!\downarrow_{\{a_{i_1}v_{i_1}, \cdots, a_{i_t}v_{i_t}\}}\!| = o((2(n-2t)-1)!!)$$ for each term on the right-hand side of (\ref{eq:terms}).  Since the right-hand side of (\ref{eq:terms}) has $O((2n)_t)$ terms, we have that 
\begin{align*}
\left| \mathcal{F} \setminus \bigcup_{i=1}^t \mathcal{F} \!\downarrow_{\{2i-1,2i\}} \right|~&\leq o((2(n-2t)-1)!!)O((2n)_{t})\\
&= o((2(n-t)-1)!!).
\end{align*}
Recalling that $|\mathcal{F}| \geq c(2(n-t)-1)!!$, by averaging, there is an edge $ij \in T^*$ such that
\[   |\mathcal{F} \!\downarrow_{ij}| \geq (c - o(1))(2(n-t)-1)!!/t.   \]
The set of all members of $\mathcal{F}$ that do not contain the edge $ij$ can be written as
\begin{align}\label{eq:terms2}
\mathcal{F} \setminus \mathcal{F} \!\downarrow_{ij}~=  \bigcup_{ k \neq j } \mathcal{F} \!\downarrow_{ik}.
\end{align}
By Lemma~\ref{lem:1cross}, we have $|\mathcal{F} \!\downarrow_{ij}| \cdot | \mathcal{F} \!\downarrow_{ik}| \leq ((2(n-t-1)-1)!!)^2$ for each term on the right-hand side of (\ref{eq:terms2}).  Since $|\mathcal{F} \!\downarrow_{ij}| \geq \Omega((2(n-t)-1)!!)$, we deduce that $|\mathcal{F} \!\downarrow_{ik}| \leq O((2(n-t-2)-1)!!)$, which gives us
 
\[ |\mathcal{F} \setminus \mathcal{F}_{ij}| = \sum_{k \neq j} |\mathcal{F} \!\downarrow_{ik}| \leq O((2(n-t-1)-1)!!).   \]

\noindent At this point we have shown that any large $t$-intersecting family $\mathcal{F}$ is almost contained within a canonically intersecting family $\mathcal{F}_{ij}$. This may seem problematic, after all, the key lemma states that any large $t$-intersecting family is almost contained within a canonically $t$-intersecting family $\mathcal{F}_T$; however, we are in the homestretch, as a simple induction on $t$ following Ellis~\cite{Ellis11} will take us the rest of the way.\\

If $t = 1$, then we are done, so let us assume that the key lemma is true for $t-1$. 
Let $\mathcal{F} \subseteq \mathcal{M}_{2n}$ be a $t$-intersecting
family of size at least $c(2(n - t)-1)!!$. We have shown there exists an edge $ij$ such that 
$$|\mathcal{F} \setminus \mathcal{F} \!\downarrow_{ij}| \leq O((2(n - t-1) - 1)!!),$$ 
which implies that
$$| \mathcal{F} \!\downarrow_{ij}| \geq  |\mathcal{F}| - O((2(n - t-1) - 1)!!).$$ 
By removing the vertices $i$ and $j$ from each member of $\mathcal{F} \!\downarrow_{ij}$, we obtain
a $(t - 1)$-intersecting family $\mathcal{F}' \subseteq \mathcal{M}_{2(n-1)}$ of perfect matchings of $K_{2n} \setminus \{i,j\}$ such that
$$ |\mathcal{F}'| \geq (c - O(1/n))(2(n - t)-1)!!.$$
For any $c' \in (0,c)$ we have 
$$|\mathcal{F}'| \geq c'(2(n - t)-1)!!$$ 
provided that $n$ is sufficiently large.  Now by the induction hypothesis, there exists a canonically $(t - 1)$-intersecting family $\mathcal{F}'_{T'}$ of perfect matchings of $K_{2n} \setminus \{i,j\}$ such that 
$$|\mathcal{F}' \setminus \mathcal{F}'_{T'}| \leq O((2(n-t-1)-1)!!).$$ 
Setting $T = T' \cup \{ij\}$, if we add the edge $ij$ to each member of $\mathcal{F}'_{T'}$, then we obtain the canonically $t$-intersecting family $\mathcal{F}_T$ of perfect matchings of $K_{2n}$.  This implies that
$$|\mathcal{F} \setminus \mathcal{F}_T| \leq O((2(n-t-1)-1)!!),$$
as desired. This finishes the proof of the key lemma and thus the proof of our main result.

\section{Concluding Remarks}\label{sec:odd}

It is natural to conjecture that similar results hold for near-perfect matchings of $K_{2n-1}$.  Without much extra effort, we can give an analogue of the first part of our main result for near-perfect matchings.

Let $\mathcal{M}_{2n-1}$ denote the collection of \emph{near-perfect matchings} of $K_{2n-1}$, equivalently, maximum matchings of $K_{2n-1}$. We may identify them with the cosets of the quotient $S_{2n-1}/H_{n-1}$. Let  $\mathcal{O}(2n-1)$ denote the irreducibles of $S_{2n-1}$ that have precisely one odd part.
The theorem below follows immediately from Pieri's rule and Theorem~\ref{thm:decomp}. 
\begin{theorem}\label{thm:decomp_odd} The space of real-valued functions over near-perfect matchings $K_{2n-1}$ admits the following decomposition into irreducibles of $S_{2n-1}$:
\[1 \uparrow^{S_{2n-1}}_{H_{n-1}} \cong \mathbb{R}[\mathcal{M}_{2n-1}] \cong  \bigoplus_{\lambda \in \mathcal{O}(2n-1)} \lambda.\]
\end{theorem}
\noindent This implies that the permutation representation of $S_{2n-1}$ acting on $\mathcal{M}_{2n-1}$ is multiplicity-free, so we have that $(S_{2n-1}, H_{n-1})$ is a symmetric Gelfand pair. We define the corresponding symmetric association scheme below. 

For each $\lambda \in \mathcal{O}(2n-1)$, the $\lambda$-associate $A_{\lambda}$ is the following matrix
\[
    (A_\lambda)_{i,j} = 
\begin{cases}
    1,& \text{if } d'(i,j) = \lambda\\
    0,              & \text{otherwise}
\end{cases}
\]
where $i,j \in \mathcal{M}_{2n-1}$ and $d'$ is a cycle type function defined as follows. Recall that the multiset union of two near-perfect matchings $m,m'$ is a collection of even cycles and precisely one path of even length.  We may represent this multiset union again as a partition $d'(m,m') = \lambda \vdash (2n-1)$ such that $\lambda$ has precisely one odd part. The odd part, say $\lambda_i$, represents the unique even path of length $\lambda_i - 1$. The collection of matrices $\{A_\lambda : \lambda \in \mathcal{O}(2n-1) \}$ forms the \emph{near-perfect matching association scheme}.

Let $\Gamma_t'$ be the near-perfect matching variant of the $t$-derangement graph, that is, $m,m' \in E(\Gamma_t')$ if $d'(m,m')$ has less than $t$ parts of size 2.
Let $\Theta_{t}$ be the subgraph of $\Gamma_t'$ whose adjacency matrix is the following sum of associates of the near-perfect matching association scheme
$$\Theta_{t} = \sum_{\lambda} A_\lambda $$
where $\lambda$ ranges over all partitions of $\mathcal{O}(2n-1)$ that have less than $t$ parts of size less than or equal to $2$.
\begin{proposition}\label{prop:isomorphic}
$\Theta_{t} \cong \Gamma_t$
\end{proposition}
\begin{proof}
Identify the vertices of $K_{2n-1}$ and $K_{2n}$ with the sets $[2n-1]$ and $[2n]$ respectively. There is a natural map $\psi : \mathcal{M}_{2n-1} \rightarrow \mathcal{M}_{2n}$ defined such that $\psi(m') = m$ where $m \in \mathcal{M}_{2n}$ is the unique perfect matching that matches the vertex $2n \in V(K_{2n})$ with the unique unmatched vertex of $m'$.  This map is a bijection such that 
$$m_1',m_2' \in E(\Theta_t) \quad \text{ if and only if } \quad \psi(m_1'),\psi(m_2') \in E(\Gamma_t)$$
for each pair $m_1',m_2' \in \mathcal{M}_{2n-1}$, which gives the desired isomorphism.
\end{proof}
\noindent Since $\Theta_t$ is a subgraph of $\Gamma_t'$, the canonically $t$-intersecting families of $\mathcal{M}_{2n-1}$, which have size $(2(n-t)-1)!!$, are also independent sets of $\Theta_t$. Proposition~\ref{prop:isomorphic} along with the results of Section~\ref{sec:pseudo} give us the following.
\begin{theorem}
Let $t \in \mathbb{N}$. If $\mathcal{F} \subseteq \mathcal{M}_{2n-1}$ is $t$-intersecting, then
$$|\mathcal{F}| \leq (2(n-t)-1)!!$$
for sufficiently large $n$ depending on $t$.
\end{theorem}
\begin{theorem}
Let $t \in \mathbb{N}$. If $\mathcal{F},\mathcal{G} \subseteq \mathcal{M}_{2n-1}$ are cross $t$-intersecting, then
$$|\mathcal{F}|\cdot|\mathcal{G}| \leq ((2(n-t)-1)!!)^2$$
for sufficiently large $n$ depending on $t$.
\end{theorem}
A similar characterization of the extremal $t$-intersecting families probably holds for near-perfect matchings, but we have not worked out these details.

Combinatorial methods of Ellis~\cite{Ellis11} in all likelihood can be used to turn the stability results of the previous section into stronger Hilton-Milner-type results that characterize the largest $t$-intersecting families of perfect matchings that are not contained in any canonically $t$-intersecting family. We have worked out such a characterization for the $t=1$ case in an unpublished note; but the proof is virtually identical to Ellis'~\cite{Ellis11}.

We have made no attempt to find a concrete $n_0(t) \in \mathbb{N}$ such that Theorem~\ref{thm:tekr} holds for all $n \geq n_0(t)$.  This is due to the fact that the $n_0(t)$ that one would obtain by walking through our proof would be quite far from the $n_0(t) = 3t/2 + 1$ conjectured by Godsil and Meagher.  We have also assumed that $t$ is independent of $n$ in several places, so it seems that a different strategy is needed to fully resolve Godsil and Meagher's conjecture. 

Finally, we believe there are other symmetric association schemes that admit orderings of its eigenspaces and associates such that the ``low eigenspaces" support the characteristic vectors of maximum independent sets in the union of its ``top associates". Indeed, many of the classic algebraic proofs of $t$-intersecting Erd\H{o}s-Ko-Rado results at a high level stem the fact that an association scheme is $P$-polynomial and $Q$-polynomial with orderings of the associates and eigenspaces such that
\[ \left(\sum_{i=0}^{f(t)} E_{i} \right)1_S = 1_S \]
for any maximum independent set $S$ of the graph $\sum_{i=0}^{f(t)-1} A_{m-i}$ (see~\cite{GodsilMeagher}).  
Note that both the perfect matching scheme and the conjugacy class association scheme of $S_n$ are not $P$-polynomial or $Q$-polynomial, yet there still exists such an ordering of their associates and eigenspaces, and this plays an essential role in the $t$-intersecting Erd\H{o}s-Ko-Rado-type bounds in~\cite{EllisFP11} and the present work. Investigating generalizations of the $P$-polynomial and $Q$-polynomial property may help understand when general association schemes have the right structure for showing $t$-intersecting Erd\H{o}s-Ko-Rado-type bounds.  
\bibliographystyle{plain}
\bibliography{../master}

\begin{thebibliography}{10}

\bibitem{BannaiI84}
E.~Bannai and T.~Ito.
\newblock {\em Algebraic combinatorics I: association schemes}.
\newblock Mathematics lecture note series. Benjamin/Cummings Pub. Co., 1984.

\bibitem{Bauer07}
F.~L. Bauer.
\newblock A remark on {S}tirling's formula and on approximations for the double
  factorial.
\newblock {\em The Mathematical Intelligencer}, 29(2):10 -- 14, 2007.

\bibitem{CST}
T.~Ceccherini-Silberstein, F.~Scarabotti, and F.~Tolli.
\newblock {\em Harmonic Analysis on Finite Groups: Representation Theory,
  Gelfand Pairs and Markov Chains}.
\newblock Cambridge Studies in Advanced Mathematics. Cambridge University
  Press, 2008.

\bibitem{Diaconis88}
Persi Diaconis.
\newblock {\em Group representations in probability and statistics}.
\newblock Institute of Mathematical Statistics Lecture Notes---Monograph
  Series, 11. Institute of Mathematical Statistics, Hayward, CA, 1988.

\bibitem{EllisFP11}
D.~Ellis, E.~Friedgut, and H.~Pilpel.
\newblock Intersecting families of permutations.
\newblock {\em J. Amer. Math. Soc.}, 24:649--682, 2011.

\bibitem{Ellis11}
David Ellis.
\newblock {S}tability for $t$-intersecting families of permutations.
\newblock {\em Journal of Combinatorial Theory, Series A}, 118(1):208 -- 227,
  2011.

\bibitem{GodsilMeagher}
C.~Godsil and K.~Meagher.
\newblock {\em Erdos-Ko-Rado Theorems: Algebraic Approaches}.
\newblock Cambridge Studies in Advanced Mathematics. Cambridge University
  Press, 2015.

\bibitem{GodsilM15}
Chris Godsil and Karen Meagher.
\newblock An algebraic proof of the {E}rd{\"o}s-{K}o-{R}ado theorem for
  intersecting families of perfect matchings.
\newblock {\em ARS MATHEMATICA CONTEMPORANEA}, 12(2):205--217, 2016.

\bibitem{KuW17}
Cheng~Yeaw Ku and Kok~Bin Wong.
\newblock Eigenvalues of the matching derangement graph.
\newblock {\em Journal of Algebraic Combinatorics}, Dec 2017.

\bibitem{Lindzey18}
N.~{Lindzey}.
\newblock {Stability for intersecting families of perfect matchings
  (submitted)}.
\newblock {\em {A}r{X}iv e-prints}, August 2018.

\bibitem{Lindzey17}
Nathan Lindzey.
\newblock Erd{\"o}s-{K}o-{R}ado for perfect matchings.
\newblock {\em European Journal of Combinatorics}, 65:130 -- 142, 2017.

\bibitem{MacDonald95}
I.G. Macdonald.
\newblock {\em Symmetric functions and Hall polynomials}.
\newblock Oxford mathematical monographs. Clarendon Press, 1995.

\bibitem{MeagherM05}
Karen Meagher and Lucia Moura.
\newblock Erd{\"o}s-{K}o-{R}ado theorems for uniform set-partition systems.
\newblock {\em Electr. J. Comb.}, 12(1):Research Paper 40, 12 pp. (electronic),
  2005.

\bibitem{Serre96}
L.L. Scott and J.P. Serre.
\newblock {\em Linear Representations of Finite Groups}.
\newblock Graduate Texts in Mathematics. Springer New York, 1996.

\bibitem{Stanley89}
Richard~P Stanley.
\newblock Some combinatorial properties of {J}ack symmetric functions.
\newblock {\em Advances in Mathematics}, 77(1):76 -- 115, 1989.

\bibitem{StanleyV201}
R.P. Stanley.
\newblock {\em Enumerative Combinatorics:}, volume~2 of {\em Cambridge Studies
  in Advanced Mathematics}.
\newblock Cambridge University Press, 2001.

\bibitem{Thrall42}
R.~M. Thrall.
\newblock On symmetrized {K}ronecker powers and the structure of the free {L}ie
  ring.
\newblock {\em American Journal of Mathematics}, 64(1):pp. 371--388, 1942.

\end{thebibliography}

\end{document}